\documentclass[11pt,a4paper]{amsart}

\usepackage[T1]{fontenc}
\usepackage[utf8]{inputenc}
\usepackage{graphicx}
\graphicspath{{figures_cropped/}}

\usepackage[final]{microtype}
\microtypesetup{expansion=false} 

\usepackage{amsmath,amssymb,amsfonts,amsthm,mathtools}

\usepackage[a4paper,margin=2.7cm]{geometry}

\usepackage[colorlinks=true,linkcolor=blue,citecolor=blue,urlcolor=blue]{hyperref}

\usepackage{graphicx}
\usepackage{tikz}

\usepackage{amsmath}
\usepackage{amssymb}
\usepackage{amsfonts}
\usepackage{amsthm}
\usepackage{mathtools}

\usepackage[a4paper,margin=2.7cm]{geometry}

\usepackage{graphicx}
\usepackage{tikz}
\usetikzlibrary{arrows.meta,decorations.markings,calc}

\usepackage{microtype}
\usepackage{enumitem}
\usepackage{hyperref}

\newtheorem{theorem}{Theorem}[section]
\newtheorem{lemma}[theorem]{Lemma}
\newtheorem{proposition}[theorem]{Proposition}

\newtheorem{claim}[theorem]{Claim}
\theoremstyle{definition}
\newtheorem{definition}[theorem]{Definition}
\newtheorem{example}[theorem]{Example}
\newtheorem{remark}[theorem]{Remark}

\title[ON MEROMORPHIC PENCILS, CUSP SINGULARITIES AND HOLOMORPHIC FOLIATIONS]{ON MEROMORPHIC PENCILS, CUSP SINGULARITIES AND HOLOMORPHIC FOLIATIONS IN THE COMPLEX PLANE}
\author{Bruno Scárdua}
\date{}

\email{bruno.scardua@gmail.com}

\keywords{holomorphic foliation; pencil; singularity; first integral; cusp 
singularity}
\date{}
\subjclass[2000]{Primary 37F75, 57R30; Secondary 32M25, 32S65.}

\thanks{Partial financial support was received from Fundação Getúlio Vargas (FGV-Rio de Janeiro), through the  The School of Applied Mathematics (EMAP-FGV).
The authors have no competing interests to declare that are relevant to the content of this article.
}

\begin{document}

\maketitle


\begin{abstract}

We study polynomial holomorphic $1$-forms in $\mathbb{C}^2$ that are homologically trivial along the fibers of meromorphic pencils of the form
$
\phi = \frac{f^p}{g^q},
$
where $f,g$ are holomorphic functions (possibly  polynomials) in general position and $(p,q)=1$.
We first establish a homological characterization of relative exactness: if a polynomial $1$-form $\Omega$ has vanishing periods along every closed path contained in the fibers $\phi_c$, then $\Omega$ decomposes as
$\Omega = a\omega_0 + dh,$ where $\omega_0 = p gdf - q f dg,$
for suitable polynomials $a$ and $h$. In the homogeneous case, degree constraints force $a$ to be constant.
We then apply this integration principle to foliations leaving invariant plane curve singularities of cusp type
\[
f^p + g^q = 0.
\]
Under a natural genericity (Morse type) condition, we prove a globalization theorem showing that homological triviality along the associated pencil implies that $\Omega$ is a polynomial cusp basic form,
\[
\Omega = d(f^p + g^q) + \lambda (p gdf - q fdg),
\qquad \lambda \in \mathbb{C}.
\]
In particular, such foliations admit Liouvillian first integrals of hypergeometric type.

Our results provide a bridge between relative cohomology, the geometry of rational pencils, and the analytic structure of cusp foliations, yielding explicit normal forms and first integrals under homological hypotheses.
\end{abstract}

\tableofcontents
\section{Introduction}
\label{section:intro}
This work is devoted to the study of pencils of holomorphic foliations. These
pencils we consider being given by meromorphic functions
$
\phi=\frac{f^p}{g^q},
$
where $f,g$ are holomorphic functions (sometimes polynomial or even homogeneous
polynomial) and $p,q\in\mathbb{N}$ are relatively prime natural numbers.
The indeterminacy points of $\phi$, i.e., the points where both $f$ and $g$
vanish, are dicritical singularities of the pencil. Near such a dicritical
singularity the pencil is described by the curves
$
\frac{f^p}{g^q}=c, \quad c\in\mathbb{C}\cup\{\infty\},
$
these are unions of cusp type centered at the dicritical singularity.
We study perturbations of pencils in dimension two.  The local case of the pencil $x/y=c$ has already been studied in \cite{ScarduaIJM}. Some  original ideas go back to \cite{Mucino}. We look for conditions that
may characterize the existence of a meromorphic first integral for the
perturbation. Special attention is paid to the case of analytic one-parameter
deformations
$\Omega_t=\omega_0+t\Omega_1,$
of order one in the complex parameter $t\in(\mathbb{C},0)$, where
$
\omega_0:=pg\,df-qf\,dg
$
and $\Omega_1$ is a holomorphic $1$-form.

Moved by this goal we revisit and look for possible variants of Ilyashenko's
PhD thesis \cite{Ilyashenko}. In his thesis Ilyashenko proves that, under some
generic conditions, a polynomial $1$-form
\[
\Omega=A(z,w)\,dz+B(z,w)\,dw, \qquad A,B\in\mathbb{C}[z,w],
\]
is exact iff its restrictions $\Omega|_{\phi_c}$ to the fibers
$
\phi_c=\{P(z,w)=c\},
$
where $P(z,w)$ is a generic (Morse type) polynomial, are all exact. Since
$\Omega|_{\phi_c}$ is already closed (because $\Omega$ and $\phi_c$ are holomorphic),
the exactness of $\Omega|_{\phi_c}$ is equivalent to its vanishing in the first
homology group of $\phi_c$. 
This last condition can be written as
$
\int_\gamma \Omega = 0,
$
for every closed path $\gamma\subset\phi_c$, $c\in\mathbb{C}\cup\{\infty\}$.

As in \cite{Ilyashenko}, a genericity condition is required. We shall need the following notion:
\begin{definition}[Morse type pair]
\label{definition:*generic}
Any pair $f,g \in \mathbb C[z,w]$ of irreducible polynomials  is said to be a  {\it Morse type pair} if:
\begin{enumerate}
\item[{\rm($\mathcal M.1$)}] the projective curves
\[
S_f:=\overline{\{f(z,w)=0\}}\subset\mathbb{CP}^2
\quad\text{and}\quad
S_g:=\overline{\{g(z,w)=0\}}\subset\mathbb{CP}^2
\]
 meet only at transverse
points.

\item[{\rm($\mathcal M.2$)}] The polynomials $f(z,w),g(z,w)$ are
{of   Morse type} in the sense that their
singularities are of Morse type and
each of them exhibits no saddle-connection. Moreover, 
at least one of the curves $S_f$ or $S_g$ does not connect two dicritical
singularities.
\end{enumerate}

\end{definition}

The notion of Morse type  pair is quite natural from the point of view of dynamical systems. With this notion we obtain the following version of Theorem~1 in \cite{Ilyashenko}.

\begin{theorem}[first integration lemma]
\label{theorem:firsintegration}
Let $f(z,w), g(z,w)$ be complex polynomials
$f,g\in\mathbb{C}[z,w]$, $p, q \in \mathbb N$ with no common factor.
Let $\Omega = A(z,w)\,dz + B(z,w)\,dw$
be a polynomial $1$-form in $\mathbb{C}^2$. 
Assume that the pair $f,g$ is Morse type. 
Also assume that $\int_\gamma \frac{\Omega}{f g}=0
$
for every closed path $\gamma\subset\phi_c$,
and every $c$, where $\phi:=f^p/g^q$.
Then there is a pair of polynomials
$a,h\in\mathbb{C}[z,w]$
such that
$\frac{\Omega}{f g}
=
a \frac{\omega_0}{fg} + dh.
$
\end{theorem}

 Ilyashenko goes deep   into the problem of persistence of a polynomial first integral
under one-parameter deformations
$
\Omega_t=dP+t\Omega_1
$
in connection with the persistence of a center type singularity for the deformation.
As a basic ingredient, it is proven in \cite{Ilyashenko} that $\Omega_t$ is exact of
the form $\Omega_t=dP_t$ provided that $\Omega_t$ vanishes identically in the first
homology of the fibers $\phi_c$ and that this is the case if a center singularity
$\xi_0$ of $dP$ persists as a center singularity $\xi_t$ of $\Omega_t$.
We prove some possible versions of the above results, replacing the polynomial
$P$ by a rational function $\phi=f^p/g^q$ or the germ of
meromorphic function  $\phi = x^p / y^q $ in a neighbourhood of the origin, $0 \in \mathbb{C}^2.$
The next result can be seen as a version of the "Destruction-of-center theorem" (Corollary~1) in \cite{Ilyashenko}.

\begin{theorem}[persistence theorem]
\label{theorem:persistence}
Let \(f,g \in \mathbb C[z,w]\) be a Morse type pair of homogeneous polynomials  and 
\(p,q \in \mathbb N\), \(\langle p,q \rangle = 1\).
Let $\Omega_t = \Omega_0 + t \Omega_1
$
be an analytic deformation of
$
\Omega_0 = p g df - q f dg
$
by polynomial 1-forms of degree
$
\deg \Omega_t = \deg f + \deg g - 1$, where $t$ is a complex parameter.

Assume that the singularity \(0\) is a dicritical singularity of \(\Omega_t\) for all \(t\).
Then \(\Omega_t\) admits a rational first integral, $\forall t$.
\end{theorem}

The connection with the theory of cusp singularities comes from the normal
form
\[
\Omega(b) = d(x^p + y^q) + b(x,y)(p y\,dx - q x\,dy)
\]
where $b(x,y)$ is a meromorphic function (\cite{Cerveau-Moussu}, \cite{CerveauLoray1998}).
The function $b(x,y)$ is holomorphic if and only if there is no other separatrix (\cite{Camacho-Sad}) than the
cusp $x^p +  y^q = 0$ and there are no saddle-nodes in the reduction of the
singularity (\cite{Cerveau-Moussu}).
We prove the existence of a (global) polynomial cusp type normal form associated with
a polynomial $1$-form having a cusp type singularity at $0 \in \mathbb{C}^2$
and vanishing in the first homology group of the fibers
$\phi_c = \left\{ {\phi(z,w)} = c \right\}.$


\begin{theorem}[first globalization theorem]
\label{theorem:firstglobalizationcusp}
Let $f,g\in\mathbb{C}[z,w]$ be homogeneous
polynomials constituting a Morse type pair.
Let
\[
\Omega = A(z,w)\,dz + B(z,w)\,dw
\]
be a polynomial $1$-form. Assume that $\Omega$ defines a cusp form in a neighborhood of the origin 
$0  \in \mathbb C^2$. 
Suppose that
\[
\int_\gamma \Omega = 0
\]
for every closed path
$
\gamma \subset
\phi_c\,, 
\quad
\forall c\in\mathbb{C}\cup\{\infty\}.
$
Then $\Omega$ is a polynomial cusp form 
\[
\Omega
=    d(f^p+g^q) + a(pgdf - q f dg)\,,
\]
for some $a\in \mathbb C[z,w]$.
\end{theorem}

The next result assures the existence of a liouvillian first integral, of hypergeometrical type, 
for polynomial forms, with a homology vanishing condition and admitting a cusp invariant affine curve:

\begin{theorem}[second globalization theorem]
\label{theorem:secondglobalizationcusp}
Let $f,g\in\mathbb{C}[z,w]$ be homogeneous
polynomials in general position and forming a Morse type pair, \, 
$p,q\in\mathbb{N}$, $\langle p,q\rangle=1$.

Let
\[
\Omega = A(z,w)\,dz + B(z,w)\,dw
\]
be a polynomial $1$-form leaving invariant  the  cusp 
$f^p + g^q=0$ and let $\deg \Omega \leq \deg (f^p + g^q)$. 
Suppose that
\[
\int_\gamma \Omega = 0
\]
for every closed path
\[
\gamma \subset
\phi_c
=
\left\{
\frac{f(z,w)^p}{g(z,w)^q}=c
\right\},
\quad
\forall c\in\mathbb{C}\cup\{\infty\}.
\]

Then $\Omega$ is a polynomial cusp basic form 
\[
\Omega
=
d(f^p + g^q)
+
\lambda(p g\,df - q f\,dg),
\qquad
\lambda \in\mathbb C.
\]
In particular, $\Omega$ admits a liouvillian hypergeometric  first integral. 

\end{theorem}

\section{The Cohomological equation}
\label{section:cohomoeq}
In this preparatory section we study the cohomological equation in the local case with meromorphic first integral of type $x^p/y^q$. More precisely,  we consider a pair of relative prime numbers $p,q\in\mathbb N$, $\langle p,q\rangle=1$, local coordinates $(x,y)$ at the origin $0\in\mathbb C^2$,
 the $1$-form $\omega_0 = p y\,dx - q x\,dy$,
and   a meromorphic $1$-form $\eta=\eta(x,y)$, with simple poles, contained in $(xy=0)$ and defined in a neighborhood of $0\in\mathbb C^2$.
Write $\eta
=
\sum_{n,m\ge1} a_{nm} x^n y^m \, dx
+
\sum_{n,m\ge1} b_{nm} x^n y^m \, dy
$
and
$
h
=
\sum_{n,m\ge0} h_{nm} x^n y^m.
$
We shall consider the {\it cohomological equation} 
\begin{equation}
\label{equation:cohomological}
\eta\wedge\omega_0 = dh\wedge\omega_0,
\end{equation}
for some meromorphic function $h(x,y)$ defined in a neighborhood of $0\in\mathbb C^2$. Similar situations have been studied in \cite{[Cerveau--Berthier]}, \cite{ScarduaIJM} and \cite{LeonScarduaJS}. 
Let us do some calculations with power series.

\[
\eta \wedge \omega_0
=
\left(
\sum_{n,m\ge1} a_{nm} x^n y^m dx
+
\sum_{n,m\ge1} b_{nm} x^n y^m dy
\right)
\wedge
(py\,dx - qx\,dy)
\]

\[
=
\sum_{n,m\ge1}
-q a_{nm} x^{n+1} y^m dx\wedge dy
-
\sum_{n,m\ge1}
p b_{nm} x^n y^{m+1} dx\wedge dy
+ (q a_{1,0}+p b_{0,1})
\]

\[
=
-
\left[
q \sum_{m\ge0}\sum_{n\ge1}
a_{n-1,m} x^n y^m dx\wedge dy
+
p \sum_{n\ge0}\sum_{m\ge1}
b_{n,m-1} x^n y^m dx\wedge dy
\right]
-
(q a_{1,0}+p b_{0,1})
\]

\bigskip

Let us write

\[
\varphi(x)
=
\sum_{m\ge0} a_{m-1,-1} x^m,
\qquad
\psi(y)
=
\sum_{m\ge0} b_{-1,m-1} y^m .
\]

\[
\eta\wedge \omega_0
=
-
\left[
q y^{-1}\varphi(x)
+
p x^{-1}\psi(y)
+
\sum_{n,m\ge0}
(q a_{n-1,m} + p b_{n,m-1}) x^n y^m
\right]
dx\wedge dy
\]

\bigskip

\[
\varphi(x)
=
a_{-1,-1}x^0
+
\sum_{m\ge1} a_{m-1,-1}x^m
=
a_{-1,-1} + \sum_{m\ge1} a_{m-1,-1}x^m
\]

\[
\psi(y)
=
b_{-1,-1}y^0
+
b_{-1,0}y^1
+
\sum_{m\ge2} b_{-1,m-1}y^m .
\]

We have that

\[
-
\left[
q y^{-1}\varphi(x)
+
p x^{-1}\psi(y)
\right] dx\wedge dy
=
(\varphi(x)dx + \psi(y)dy)
\wedge
\left(
\frac{p\,dx}{x} + \frac{q\,dy}{y}
\right)
\]

\[
=
(\varphi(x)dx + \psi(y)dy)
\wedge
d(x^p y^q).
\]

However,

\[
\varphi(x) = a_{-1,-1} + \Phi'(x)
\quad \text{for some } \Phi(x),
\]

and

\[
\psi(y) = b_{-1,-1} + \Psi'(y)
\quad \text{for some } \Psi(y).
\]

\bigskip

We may therefore write

\[
-
(\eta\wedge \omega_0)
=
+
\sum_{n,m\ge0}
(q a_{n-1,m} + p b_{n,m-1}) x^n y^m dx\wedge dy
\]

\[
+
(\Phi'(x)dx + \Psi'(y)dy)
\wedge
\frac{d(x^p y^q)}{x^p y^q}
\]

\[
+
(q a_{1,0}+p b_{0,1}) dx\wedge dy
\]

\[
+
\left(
q y^{-1} a_{-1,-1}
+
p x^{-1} b_{-1,-1}
\right) dx\wedge dy .
\]

On the other hand,

\[
h = \sum_{n,m\ge0} h_{nm} x^n y^m
\]

\[
dh =
\sum_{n,m\ge0} n h_{nm} x^{n-1} y^m dx
+
\sum_{n,m\ge0} m h_{nm} x^n y^{m-1} dy
\]

\[
dh \wedge \omega_0
=
dh \wedge (p y\,dx - q x\,dy)
\]

\[
=
\sum_{n,m\ge0}
(-nq\, h_{nm} x^n y^m
-
mp\, h_{nm} x^n y^m) dx\wedge dy
\]

\[
=
-
\sum_{n,m\ge0}
(nq + mp) h_{nm} x^n y^m dx\wedge dy .
\]

Therefore, $\eta \wedge \omega_0 = dh \wedge \omega_0$ is equivalent to
\[
\sum (q a_{n-1,m} + p b_{n,m-1}) x^n y^m dx\wedge dy
=
\sum (nq+mp) h_{n,m} x^n y^m dx\wedge dy
\]

and this is equivalent to

\[
q a_{n-1,m} + p b_{n,m-1} = (nq+mp) h_{n,m}.
\]

This last equation has solution

\[
h_{n,m} = \frac{q a_{n-1,m} + p b_{n,m-1}}{nq+mp}
\]

if and only if
\[
q a_{n-1,m} + p b_{n,m-1} = 0
\quad \text{every time that} \quad
nq+mp=0.
\]

\begin{definition}[resonance] We shall refer to the equation $nq+mp=0$ as the {\it resonance equation}
and to a pair $(n,m)$ satisfying this equation as {\it a resonance}  or a {\it resonant pair}.
\end{definition}

Next we study the possible resonances, according to $p,q$.

\noindent{\bf 1st case} : $\quad p>1,\; q>1.$

The resonance equation $nq+mp=0$ gives us $nq=-mp$ and then $n,m$ have opposite
signals.

From $nq=-mp$ we get $p \mid nq \Rightarrow p \mid n \Rightarrow n=kp$ for some $k\in\mathbb{N}$.

From $nq=-mp$ we have $kp\,q=-mp \Rightarrow m=-kq$.

\[
\Rightarrow (n,m)=k(p,-q).
\]

For
\[
\begin{cases}
n\ge -1\\
m\ge -1
\end{cases}
\]
we obtain $k=0$ so that the only resonance occurs for the pair $(n,m)=(0,0)$.

\noindent{\bf 2nd case} : $\quad p=1<q.$

We always have the trivial solution $(n,m)=(0,0)$.

This trivial solution is associated with a logarithmic part

\[
a_{n-1,m} x^{n-1}y^m dx + b_{n,m-1} x^n y^{m-1} dy
=
a_{-1,0}\frac{dx}{x} + b_{0,-1}\frac{dy}{y}.
\]

Let us look for the other solutions.

From $nq=-mp=-m$ and from $n\ge -1, m\ge -1$
we obtain $n>0$ and $m<0$. Therefore $m=-1$ so that $n=q$.

The pair $(n,m)=(-1,q)$ is associated with the term

\[
a_{n-1,m} x^{n-1}y^m dx + b_{n,m-1} x^n y^{m-1} dy
\]

\[
=
a_{-2,q} x^{-2} y^q dx + b_{-1,q-1} x^{-1} y^{q-1} dy.
\]

Since $a_{-2,q}=0$

\[
=
b_{-1,q-1} x^{-1} y^{q-1} dy
=
b_{-1,q-1}\frac{y^{q-1}}{x} dy.
\]

\noindent{\bf 3rd case} : $\quad q=1<p.$

Analogously to the 2nd case we obtain $(n,m)\in\{(0,0),(p,-1)\}$.

These resonances are associated with the logarithmic part
\[
a_{-1,0}\frac{dx}{x} + b_{0,-1}\frac{dy}{y}
\]
and the part
\[
a_{p-1,-1} \frac{x^{p-1}}{y} dx.
\]

\noindent{\bf 4th case:} $\quad p=q=1$

The resonance equation becomes
\[
nq+mp=0 \iff n+m=0.
\]

Since $n,m\ge -1$ the only possible solutions are given by
\[
(n,m)\in\{(0,0),(1,-1),(-1,1)\}
\]
and these are associated to the logarithmic part
\[
a_{-1,0}\frac{dx}{x} + b_{0,-1}\frac{dy}{y}
\quad \text{for } (n,m)= (0,0).
\]

For $(n,m)=(1,-1)$ terms

\[
a_{n-1,m} x^{n-1}y^m dx + b_{n,m-1} x^n y^{m-1} dy
\]

\[
=
a_{0,-1}\frac{dx}{y} + b_{1,-2} x y^{-2} dy
=
a_{0,-1}\frac{dx}{y}
\]
because we are working with order of the poles $\le 1$.

For $(n,m)=(-1,1)$ we have

\[
a_{n-1,m} x^{n-1}y^m dx + b_{n,m-1} x^n y^{m-1} dy
\]

\[
=
a_{-2,1} x^{-2} y dx + b_{-1,0}\frac{dy}{x}
=
b_{-1,0}\frac{dy}{x}.
\]

\begin{definition}[resonant part]
Given the $1$-form $\eta$ we shall denote by $R(\eta)$ the {\it resonant part} of $\eta$
which is the sum of monomials

\[
a_{n-1,m} x^{n-1}y^m dx + b_{n,m-1} x^n y^{m-1} dy
\]

where $(n,m)$ runs through the resonant pairs $(n,m)$:

\[
R(\eta) =
\sum_{nq+mp=0}
\left(
a_{n-1,m} x^{n-1}y^m dx + b_{n,m-1} x^n y^{m-1} dy
\right).
\]
\end{definition}

\subsection{Non-logarithmic resonances}
We now investigate further the possible resonances in the non-logarithmic part of the 
form $\eta$ above. Given the $1$-form
\[
\eta=\sum a_{nm}x^n y^m dx + \sum b_{nm}x^n y^m dy
\]

we define

\[
\Phi(x):=\sum_{m\ge 0} a_{m-1,-1} x^{m-1}y^{-1} dx
\]

and

\[
\Psi(y):=\sum_{m\ge 0} b_{-1,m-1} x^{-1}y^{m-1} dy.
\]

Let us investigate the possible resonances in $\Phi(x)$ and $\Psi(y)$.

\noindent{\bf 1st case} : $\quad p>1,\; q>1.$

The resonance equation is $qn+mp=0$ and the coefficient condition is
\[
q a_{m-1,m} + p b_{m,m-1} = 0.
\]

For the series that defines $\Phi(x)$ we have

\[
q(m-1)+(-1)p=0
\iff
q(m-1)=p.
\]

\[
\Rightarrow p\mid(m-1)
\Rightarrow m-1=kp
\Rightarrow qkp=p
\Rightarrow qk=1
\Rightarrow \text{contradiction}.
\]

Thus there are no resonances in $\Phi(x)$.

Let us check what happens for $\Psi(y)$.

For the series that defines $\Psi(y)$ we have
\[
q(-1)+(m-1)p=0
\]
and again we reach a contradiction.

Therefore, also in this case there are no resonances in $\Psi(y)$.

\noindent{\bf 2nd case} : $\quad p=1<q.$

$\Phi(x)$ has resonances given by

\[
q(m-1)+(-1)p=0
\Rightarrow q(m-1)=1
\Rightarrow q=1
\]

and again we reach a contradiction,
and therefore no resonances for $\Phi(x)$.

$\Psi(y)$ has resonances given by

\[
q(-1)+(m-1)p=0
\Rightarrow -q+m-1=0
\Rightarrow m=q+1.
\]

This gives the resonant term

\[
b_{-1,q-1} x^{-1}y^{q-1}dy
=
b_{-1,q-1}\frac{y^{q-1}}{x}dy.
\]

\noindent{\bf 3rd case} : $\quad q=1<p.$

Analogously to the 2nd case we have resonances in $\Phi(x)$ given by

\[
a_{p-1,-1}\frac{x^{p-1}}{y}dx
\]

and no resonances in $\Psi(y)$.

\noindent{\bf 4th case} : $\quad p=q=1.$

The resonances in $\Phi(x)$ are given by terms

\[
a_{m-1,m-1}x^{m-1}y^{-1}dx
\]

and

\[
q(m-1)+(-1)p=0
\Rightarrow m-1-1=0
\Rightarrow m=1.
\]

This gives the resonant term

\[
a_{0,-1}x^{0}y^{-1}dx
=
a_{0,-1}\frac{dx}{y}.
\]

The resonances in $\Psi(y)$ are given by

\[
q(-1)+p(m-1)=0
\Rightarrow -1+m-1=0
\Rightarrow m=1,
\]

giving the resonant term

\[
b_{-1,0}x^{-1}y^{0}dy
=
b_{-1,0}\frac{dy}{x}.
\]

All these resonances have already been identified in our previous study.

\subsection{Final normal form}

From what we have seen above we have:
\begin{proposition} Let $\eta, \omega_0, p,q$ be as above. As for what concerns the 
cohomology equation $ \eta \wedge \omega_0 = dh \wedge \omega_0$ we have the following cases:

\noindent{1st case} : $\quad p>1,q>1 \Rightarrow$ the only resonant pair is $(n,m)=(0,0)$

\[
R(\eta)
=
a_{-1,0} x^{-1}y^0 dx + b_{0,-1} x^0 y^{-1} dy
=
a_{-1,0}\frac{dx}{x} + b_{0,-1}\frac{dy}{y}.
\]

\noindent{2nd case} : $\quad p=1<q \Rightarrow$ the resonant pairs are
\[
(n,m)\in\{(0,0),(-1,q)\}
\]

\[
R(\eta)
=
a_{-1,0}\frac{dx}{x} + b_{0,-1}\frac{dy}{y}
+
a_{-2,q} x^{-2} y^q dx
+
b_{-1,q-1} x^{-1} y^{q-1} dy.
\]

$\Rightarrow R(\eta)
=
a_{-1,0}\frac{dx}{x}
+
b_{0,-1}\frac{dy}{y}
+
b_{-1,q-1}\frac{y^{q-1}}{x}dy.$

\noindent{3rd case} : $\quad q=1<p \Rightarrow$ the resonant pairs are
\[
(n,m)\in\{(0,0),(p,-1)\}
\]

\[
R(\eta)
=
a_{-1,0}\frac{dx}{x}
+
b_{0,-1}\frac{dy}{y}
+
a_{p-1,-1}\frac{x^{p-1}}{y}dx.
\]

\noindent{4th case} : $\quad p=q=1 \Rightarrow$ the resonant pairs are
\[
(n,m)\in\{(0,0),(1,-1),(-1,1)\}
\]

\[
R(\eta)
=
a_{-1,0}\frac{dx}{x}
+
b_{0,-1}\frac{dy}{y}
+
a_{0,-1}\frac{dx}{y}
+
b_{1,-2}\frac{x\,dy}{y^2}
+
a_{-2,1}\frac{y\,dx}{x^2}
+
b_{-1,0}\frac{dy}{x}.
\]

\[
\Rightarrow R(\eta)
=
a_{-1,0}\frac{dx}{x}
+
b_{0,-1}\frac{dy}{y}
+
a_{0,-1}\frac{dx}{y}
+
b_{-1,0}\frac{dy}{x}.
\]
\end{proposition}

Proceeding with the above notation, by construction the $1$-form $\tilde{\eta} := \eta - R(\eta)$ has no resonances and therefore
we can solve the cohomological equation:
\[
\tilde{\eta}\wedge \omega_0 = dh \wedge \omega_0
\]
has a solution $h$.

Thus, according to this remark and following the construction of
$\tilde{\eta}$ starting from $\eta$ we obtain:

\noindent{1st case} : $\quad (p>1,q>1)$ :

\[
\tilde{\eta}
=
\eta
-
a_{-1,0}\frac{dx}{x}
-
b_{0,-1}\frac{dy}{y}.
\]

\noindent{2nd case} : $\quad (1=p<q) :$

\[
\tilde{\eta}
=
\eta
-
b_{-1,q-1}\frac{y^{q-1}}{x}dy
-
a_{-1,0}\frac{dx}{x}
-
b_{0,-1}\frac{dy}{y}.
\]

\noindent{3rd case} : $\quad (q=1<p) :$

\[
\tilde{\eta}
=
\eta
-
a_{p-1,-1}\frac{x^{p-1}}{y}dx
-
a_{-1,0}\frac{dx}{x}
-
b_{0,-1}\frac{dy}{y}.
\]

\noindent{4th case} : $\quad (p=q=1) :$

\[
\tilde{\eta}
=
\eta
-
a_{-1,0}\frac{dx}{x}
-
b_{0,-1}\frac{dy}{y}
-
a_{0,-1}\frac{dx}{y}
+
b_{-1,0}\frac{dy}{x}.
\]

 Since $\tilde{\eta}:=\eta-R(\eta)$ satisfies
\[
\tilde{\eta}\wedge \omega_0 = dh\wedge \omega_0,
\]
by Saito-de Rham division lemma \cite{[Saito]} we obtain

\[
\tilde{\eta} = dh + a \omega_0
\]
for some holomorphic function $a(xy)$.

Using this combined with the description of $R(\eta)$ we obtain:

\begin{proposition} If $\eta,\omega_0$ are as above then there are functions
$a,h\in\mathcal{O}_2$ such that:
\begin{itemize}
\item[{\rm(1)}] if $p>1,q>1$ then
\[
\eta
=
a_{-1,0}\frac{dx}{x}
+
b_{0,-1}\frac{dy}{y}
+
a\omega_0
+
dh.
\]

\item[{\rm(2)}] if $1=p<q$ then
\[
\eta
=
a_{-1,0}\frac{dx}{x}
+
b_{0,-1}\frac{dy}{y}
+
b_{-1,q-1}\frac{y^{q-1}}{x}dy
+
a\omega_0
+
dh.
\]

\item[{\rm(3)}] if $1=q<p$ then
\[
\eta
=
a_{-1,0}\frac{dx}{x}
+
b_{0,-1}\frac{dy}{y}
+
a_{p-1,-1}\frac{x^{p-1}}{y}dx
+
a\omega_0
+
dh.
\]

\item[{\rm(4)}] if $p=q=1$ then
\[
\eta
=
a_{-1,0}\frac{dx}{x}
+
b_{0,-1}\frac{dy}{y}
+
b_{-1,0}\frac{dy}{x}
+
a_{0,-1}\frac{dx}{y}
+
a\omega_0
+
dh.
\]
\end{itemize}
\end{proposition}

\section{Homological conditions}
\label{section:homoleq}
In this section we study the homological conditions associated to the local pencil
\[
\frac{x^p}{y^q}=c,\qquad c\in\mathbb{P}^1.
\]

Let us fix some notation:
given $(x_0,y_0)\in\mathbb{C}^2$ we put

\[
x(t)=x_0 e^{2\pi i q t},
\qquad
y(t)=y_0 e^{2\pi i p t},
\qquad
t\in[0,1].
\]

Then

\[
\frac{(x(t))^p}{(y(t))^q}
=
\frac{x_0^p}{y_0^q}.
\]

Put $\gamma(t)=\gamma_{(x_0,y_0)}(t)=(x(t),y(t))$.
Let $\phi:=x^p/y^q$ and put
\[
\phi_c=\left\{\frac{x^p}{y^q}=c\right\}
\]
for $c\in\mathbb{C}\cup\{\infty\}$.
Then $\gamma\subset \phi_c$ for $c=\dfrac{x_0^p}{y_0^q}$.


\subsection{Lemmata}

In this part we compute some  line integrals associated to the resonant terms in the 
cohomological equation \eqref{equation:cohomological}.
We consider 

\[
\delta(t) = (x_0 e^{2\pi i q t},\, y_0 e^{2\pi i p t}), \quad 0 \le t \le 1.
\]

\[
\int_\gamma \frac{y^{q-1}}{x^p} dy
=
\int_0^1
\frac{y_0^{q-1} e^{2\pi i p (q-1)t} \cdot 2\pi i p\, y_0 e^{2\pi i p t}}
{x_0^p e^{2\pi i q p t}} dt
\]

\[
=
2\pi i p \frac{y_0^q}{x_0^p}
\int_0^1
e^{2\pi i p q t}
e^{-2\pi i q p t} dt
=
2\pi i p \frac{y_0^q}{x_0^p}
\int_0^1 dt.
\]

For \(p=1\) we have

\[
\int_\gamma \frac{y^{q-1}}{x} dy
=
2\pi i \frac{y_0^q}{x_0}.
\]

Similarly, for \(q=1\) we have

\[
\int_\gamma \frac{x^{p-1}}{y} dx
=
2\pi i \frac{x_0^p}{y_0}.
\]

For \(p=q=1\):

\[
\int \frac{dx}{y}
=
\int_0^1
\frac{2\pi i x_0 e^{2\pi i t}}
{y_0 e^{2\pi i t}} dt
=
2\pi i \frac{x_0}{y_0}
\int_0^1 dt
=
2\pi i \frac{x_0}{y_0}.
\]

\[
\int \frac{dy}{x}
=
2\pi i \frac{y_0}{x_0}.
\]

\begin{lemma}

Let
\[
R := \lambda \frac{df}{f} + \mu \frac{dg}{g}
+ c \frac{g^{q-1}}{f} dg + \tilde c \frac{f^{p-1}}{g} df
+ \alpha \frac{df}{g} + \beta \frac{dg}{f}.
\]

Assume that
\[
\int_{\gamma_c} R = 0, \quad \forall \gamma_c \subset \phi_c.
\]

Then
\[
c = \tilde c = \alpha = \beta = 0
\quad \text{and} \quad
\lambda q + \mu p = 0.
\]
\end{lemma}
\begin{proof}

Since \(f,g\) are in general position we may choose an embedding
\[
\delta : (\mathbb C^2,0) \to (\mathbb C^m,0)
\]
such that
\[
f^* = f \circ \delta,
\quad
g^* = g \circ \delta
\]
form a system of local coordinates
\[
(f^*, g^*) = (x,y).
\]

Then from \(\int_\gamma R = 0\) we obtain

\[
\int_{\gamma_c^*} R^* = 0
\]

where $R^*=\delta ^* R$, \, \(\gamma_c^*\) is any closed path in
\[
\phi_c^* = \{ \phi^* = c \}
\]
for
\[
\phi^* = x^p / y^q.
\]

Using the computations before this lemma we have

\[
0 = \lambda q + \mu p
+ c \, 2\pi i p \frac{y_0^q}{x_0^p}
+ \tilde c \, 2\pi i q \frac{x_0^p}{y_0^q}
+ \alpha \, 2\pi i \frac{y_0}{x_0}
+ \beta \, 2\pi i \frac{x_0}{y_0}, \forall x_0,y_0.
\]

This implies
\[
c = \tilde c = \alpha = \beta = 0
\quad \text{and} \quad
\lambda q + \mu p = 0.
\]
\end{proof}


    \subsection{Homology vanishing and normal forms}

For
\[
\eta=\sum_{n,m\ge1} a_{nm}x^n y^m dx + \sum_{n,m\ge1} b_{nm}x^n y^m dy
\]

 we obtain

\[
\gamma^*\eta
=
\sum_{n,m\ge 1}
2\pi i\, q\, a_{nm}\,
x_0^n e^{2\pi i q n t}\,
y_0^m e^{2\pi i p m t}\,
x_0 e^{2\pi i q t}\, dt
+
\sum_{n,m\ge 1}
2\pi i\, p\, b_{nm}\,
x_0^n e^{2\pi i q n t}\,
y_0^m e^{2\pi i p m t}\,
y_0 e^{2\pi i p t}\, dt
\]

\[
=
2\pi i \sum_{n,m\ge 1}
q a_{nm} x_0^{n+1} y_0^m
e^{2\pi i(q(n+1)+pm)t} dt
+
2\pi i \sum_{n,m\ge 1}
p b_{nm} x_0^{n} y_0^{m+1}
e^{2\pi i(qn+p(m+1))t} dt
\]

\[
=
2\pi i \sum_{n,m\ge 0}
q a_{n-1,m} x_0^n y_0^m
e^{2\pi i(qn+pm)t} dt
+
2\pi i \sum_{n,m\ge 0}
p b_{n,m-1} x_0^n y_0^m
e^{2\pi i(qn+pm)t} dt .
\]

Hence
\[
\int_{\gamma_c} \eta = 0
\quad \text{for } c=\phi(x_0,y_0),
\]
we get

\[
0
=
\sum_{n,m\ge 0}
\int_0^1
q a_{n-1,m} e^{2\pi i(qn+pm)t} x_0^n y_0^m dt
+
\sum_{n,m\ge 0}
\int_0^1
p b_{n,m-1} x_0^n y_0^m
e^{2\pi i(qn+pm)t} dt
\]

\[
=
\sum_{n,m\ge 0}
\int_0^1
(q a_{n-1,m}+p b_{n,m-1})
x_0^n y_0^m
e^{2\pi i(qn+pm)t} dt
\]

\[
+
\sum_{m\ge 0}
\int_0^1
q a_{n-1,-1}
e^{2\pi i(qn-p)t}
x_0^n y_0^{-1} dt
+
\sum_{m\ge 0}
\int_0^1
p b_{-1,m-1}
x_0^{-1} y_0^m
e^{2\pi i(-q+mp)t} dt .
\]

Let us calculate the integrals separately:

\[
\sum_{n,m\ge 0}
\int_0^1
(q a_{n-1,m}+p b_{n,m-1})
x_0^n y_0^m
e^{2\pi i(qn+pm)t} dt
\]

\[
=
\sum_{\substack{n,m\ge 0\\ qn+pm\neq 0}}
+
\sum_{\substack{n,m\ge 0\\ qn+pm=0}} .
\]

Observe that

if $qn+pm=0$ then
\[
\int_0^1 e^{2\pi i(qn+pm)t} dt = 1,
\]

if $qn+pm\neq 0$ then
\[
\int_0^1 e^{2\pi i(qn+pm)t} dt
=
\frac{1}{2\pi i(qn+pm)}
\left[e^{2\pi i(qn+pm)t}\right]_0^1
=
0.
\]

thus we have

\[
\int_{\gamma_c} \eta
=
\sum_{\substack{n,m\ge 0\\ qn+pm=0}}
(q a_{n-1,m}+p b_{n,m-1})
x_0^n y_0^m
\]

\[
+
\sum_{m\ge 0}
\int_0^1
q a_{n-1,-1}
e^{2\pi i(qn-p)t}
x_0^n y_0^{-1} dt
\]

\[
+
\sum_{m\ge 0}
\int_0^1
p b_{-1,m-1}
x_0^{-1} y_0^m
e^{2\pi i(-q+mp)t} dt .
\]

where

\[
\sum_{m\ge 0}
\int_0^1
q a_{n-1,-1}
e^{2\pi i(qn-p)t}
x_0^n y_0^{-1} dt
=
\sum_{m\ge 0}
q a_{n-1,-1} x_0^n y_0^{-1}
\int_0^1 e^{2\pi i(qn-p)t} dt .
\]

Observe that

\[
\int_0^1 e^{2\pi i(qn-p)t} dt
=
\begin{cases}
0 & \text{if } qn\neq p,\\
1 & \text{if } qn=p.
\end{cases}
\]

Moreover $p=qn \Rightarrow$
\[
\text{if } q\ge 2 \text{ then } q\mid p \text{ and } n=1,
\quad
\text{or, if } q=1, \text{ then } n=p.
\]

Analogously,

\[
\sum_{m\ge 0}
\int_0^1
p b_{-1,m-1}
x_0^{-1} y_0^m
e^{2\pi i(-q+mp)t} dt
=
\sum_{m\ge 0}
p b_{-1,m-1} x_0^{-1} y_0^m
\int_0^1 e^{2\pi i(-q+mp)t} dt .
\]

Observe that

\[
\int_0^1 e^{2\pi i(-q+mp)t} dt
=
\begin{cases}
0 & \text{if } -q+mp\neq 0,\\
1 & \text{if } q=mp.
\end{cases}
\]

Moreover $q=mp \Rightarrow$
\[
\text{if } p\ge 2 \text{ then } m=1 \text{ and } q=p,
\quad
\text{because } (p,q)=1,
\quad
\text{or, if } p=1, \text{ then } m=q.
\]

Thus we have different cases to consider:

Case $p\ge 1,\; q\ge 1$:

We have, by the calculations above,

\[
\int_{\gamma_c}\eta
=
\sum_{\substack{n,m\ge 0\\ qn+pm=0}}
(q a_{n-1,m}+p b_{n,m-1}) x_0^n y_0^m
\]

\[
+
q a_{0,-1} x_0 y_0^{-1}
+
p b_{-1,0} x_0^{-1} y_0 .
\]

and then

\[
\int_{\gamma_c}\eta=0 \quad \forall \gamma_c
\quad \text{iff}
\]

\[
\begin{cases}
q a_{n-1,m}+p b_{n,m-1}=0
& \forall n,m\ge 0 \text{ such that } qn+pm=0,\\[6pt]
a_{0,-1}=0,\quad b_{-1,0}=0.
\end{cases}
\]

Case $p=q=1$ :

We have for the integrals

\[
\int_0^1 e^{2\pi i(qm-p)t}dt
=
\int_0^1 e^{2\pi i(m-1)t}dt
=
\begin{cases}
0 & \text{if } m\neq 1,\\
1 & \text{if } m=1.
\end{cases}
\]

Moreover

\[
\int_0^1 e^{2\pi i(-q+mp)t}dt
=
\int_0^1 e^{2\pi i(m-1)t}dt
=
\begin{cases}
0 & \text{if } m\neq 1,\\
1 & \text{if } m=1.
\end{cases}
\]

Then we have

\[
\int_{\gamma_c}\eta=0,\ \forall \gamma_c
\quad \text{if and only if}
\]

\[
\begin{cases}
q a_{n-1,m}+p b_{n,m-1}=0
& \forall n,m\ge0 \text{ such that } n+m=0,
\text{ i.e. } (n,m)\in\{(0,0),(-1,1),(1,-1)\},\\[6pt]
a_{1,-1}=0,\quad b_{-1,1}=0.
\end{cases}
\]

Summarizing we have:

\begin{proposition}
\label{proposition:homology}

Let $\eta,\omega_0$ be as above.

The following conditions are equivalent:

\noindent{\rm(i)}
\[
\int_\gamma \eta =0,\ \forall \gamma=\gamma_{(x_0,y_0)}\subset\phi_c,\ \forall c.
\]

\noindent{\rm(ii)}
\[
\exists\, a,h\in\mathcal{O}_2
\quad \text{such that}
\quad
\eta=a\omega_0+dh.
\]
\noindent{\rm(iii)}
\[
\exists\, \tilde a,\tilde h\in\mathcal{O}_2
\quad \text{such that}
\quad
\eta=\tilde a\frac{\omega_0}{xy}+d \tilde h.
\]

\end{proposition}
\begin{proof}
\noindent{ $(i)\implies(ii)$}. 
Assume that $\int_\gamma\eta=0,\ \forall\gamma\subset\phi_c,\ \forall c$.

Choose
\[
\gamma=\gamma_{(x_0,y_0)}
=
(x_0e^{2\pi iqt},y_0e^{2\pi ipt})
\]

and write

\[
\eta=\sum_{n,m\ge1}(a_{nm}x^ny^m dx+b_{nm}x^ny^m dy).
\]

Then

\[
\gamma^*\eta
=
\sum_{n,m\ge1}
a_{nm}x_0^ne^{2\pi iqnt}y_0^me^{2\pi ipmt}(2\pi iq)e^{2\pi iqt}dt
\]

\[
+
\sum_{n,m\ge1}
b_{nm}x_0^ne^{2\pi iqnt}y_0^me^{2\pi ipmt}(2\pi ip)e^{2\pi ipt}dt
\]

\[
=
\sum_{n,m\ge1}
a_{nm}x_0^{n+1}y_0^m
e^{2\pi it[(n+1)q+mp]}dt
\]

\[
+
\sum_{n,m\ge1}
b_{nm}x_0^ny_0^{m+1}
e^{2\pi it[nq+(m+1)p]}dt.
\]

Therefore

\[
\gamma^*\eta
=
2\pi i
\sum_{n,m\ge0}
q a_{n-1,m}x_0^ny_0^m
e^{2\pi it[nq+mp]}dt
\]

\[
+
2\pi i
\sum_{n,m\ge0}
p b_{n,m-1}x_0^ny_0^m
e^{2\pi it[nq+mp]}dt
\]

\[
=
2\pi i
\sum_{n,m\ge0}
(q a_{n-1,m}+p b_{n,m-1})
x_0^ny_0^m
e^{2\pi it[nq+mp]}dt
\]

\[
+
2\pi i
\sum_{m\ge0}
q a_{m-1,-1}
x_0^my_0^{-1}
e^{2\pi it[mq-p]}dt
\]

\[
+
2\pi i
\sum_{m\ge0}
p b_{-1,m-1}
x_0^{-1}y_0^m
e^{2\pi it[-q+mp]}dt.
\]

Now we make a few remarks:

\[
\int_0^1 e^{2\pi it[mq-p]}dt
=
\begin{cases}
0 & \text{if } mq\neq p,\\
1 & \text{if } mq=p.
\end{cases}
\]

If $p=mq$ then $\langle p,q\rangle=q$ so that
$q=1$ and then $m=p$.

This shows that

\[
\int_0^1 e^{2\pi it[mq-p]}dt=0
\quad \text{except for the case } q=1 \text{ and } p=m.
\]

In particular $q\ge2\Rightarrow
\int_0^1 e^{2\pi it[mq-p]}dt=0,\ \forall m\ge0.
$

While, if $q=1$,

\[
\int_0^1
\sum_{m\ge0}
q a_{m-1,-1}
x_0^my_0^{-1}
e^{2\pi it[mq-p]}dt
=
q a_{p-1,-1}x_0^py_0^{-1}
=
a_{p-1,-1}x_0^py_0^{-1}.
\]

We continue.

\[
\int_0^1 e^{2\pi it[mp-q]}dt
=
\begin{cases}
0 & \text{if } mp\neq q,\\
1 & \text{if } mp=q.
\end{cases}
\]

If $q=mp$ then $\langle p,q\rangle=p=1$ so that
$p=1$ and $m=q$.

This shows that

\[
\int_0^1 e^{2\pi it[mp-q]}dt=0
\quad \text{except for the case } p=1 \text{ and } m=q.
\]

In particular $p\ge2\Rightarrow
\int_0^1
\sum_{m\ge0}
p b_{-1,m-1}
x_0^{-1}y_0^m
e^{2\pi it[mp-q]}dt
=0.
$

While if $p=1$,

\[
\int_0^1
\sum_{m\ge0}
p b_{-1,m-1}
x_0^{-1}y_0^m
e^{2\pi it[mp-q]}dt
=
p b_{-1,q-1}x_0^{-1}y_0^q
=
b_{-1,q-1}x_0^{-1}y_0^q.
\]

Finally, our last computation is: for $m,n\ge0$

\[
\int_0^1 e^{2\pi it(mq+mp)}dt
=
\begin{cases}
0 & \text{if } mq+mp\neq 0,\\
1 & \text{if } mq+mp=0.
\end{cases}
\]

so that

\[
\int_0^1
\sum_{n,m\ge0}
(q a_{n-1,m}+p b_{n,m-1})
x_0^n y_0^m
e^{2\pi it[mq+mp]}dt
=
\sum_{\substack{n,m\ge0\\ mq+mp=0}}
(q a_{n-1,m}+p b_{n,m-1})
x_0^n y_0^m.
\]

Using all the above computations we obtain:

\[
\int_{\gamma_c}\eta
=
2\pi i
\sum_{\substack{n,m\ge0\\ mq+mp=0}}
(q a_{n-1,m}+p b_{n,m-1})
x_0^n y_0^m
+
R(x_0,y_0),
\]

where:

• $p>1,\; q>1 \Rightarrow R(x_0,y_0)=0.$

• $p=1 \Rightarrow
R(x_0,y_0)=2\pi i\, b_{-1,q-1} x_0^{-1} y_0^q.$

• $q=1 \Rightarrow
R(x_0,y_0)=2\pi i\, a_{p-1,-1} x_0^p y_0^{-1}.$

Therefore the hypothesis
\[
\int_\gamma \eta=\int_{\gamma_c}\eta=0,
\quad \forall\gamma\subset\phi_c,
\]
implies on each case:

• $p>1,\; q>1 \Rightarrow
q a_{n-1,m}+p b_{n,m-1}=0
\quad \forall n,m\ge0 \text{ such that } mq+mp=0.$

• $p=1<q \Rightarrow
q a_{n-1,m}+p b_{n,m-1}=0
\quad \forall n,m\ge0 \text{ such that } mq+mp=0
\quad \text{AND ALSO }
b_{-1,q-1}=0.$

• $q=1<p \Rightarrow
q a_{n-1,m}+p b_{n,m-1}=0
\quad \forall n,m\ge0 \text{ such that } mq+mp=0
\quad \text{AND ALSO }
a_{p-1,-1}=0.$

We consider now the case $p=q=1$:

Again, we have

\[
\int_0^1 e^{2\pi it[mq-p]}dt
=
\int_0^1 e^{2\pi it[m-1]}dt
=
\begin{cases}
0 & \text{if } m\neq1,\\
1 & \text{if } m=1.
\end{cases}
\]

this implies

\[
\int_0^1
\sum_{m\ge0}
q a_{m-1,-1}
x_0^m y_0^{-1}
e^{2\pi it[mq-p]}dt
=
q a_{0,-1} x_0 y_0^{-1}.
\]

Analogously,

\[
\int_0^1
\sum_{m\ge0}
p b_{-1,m-1}
x_0^{-1} y_0^m
e^{2\pi it[mp-q]}dt
=
p b_{-1,0} x_0^{-1} y_0.
\]

Thus in this case $p=q=1$ we obtain

\[
\int_{\gamma_c}\eta
=
2\pi i
\sum_{\substack{n,m\ge0\\ mq+mp=0}}
(q a_{n-1,m}+p b_{n,m-1})
x_0^n y_0^m
\]

\[
+
2\pi i
\left[
q a_{0,-1} x_0 y_0^{-1}
+
p b_{-1,0} x_0^{-1} y_0
\right].
\]

The conclusion is

\[
\int_\gamma\eta=0,
\quad \forall\gamma\subset\phi_c,\ \forall c
\]

\[
\Longleftrightarrow
\begin{cases}
q a_{n-1,m}+p b_{n,m-1}=0
& \forall n,m\ge0 \text{ such that } mq+mp=0,\\[6pt]
a_{0,-1}=0,\quad b_{-1,0}=0.
\end{cases}
\]

Using the above computations we conclude that
(1) $\Rightarrow$ the conditions for the solution
of the cohomological equation are verified,
i.e., (i) $\Rightarrow$ (ii).

\noindent{$(ii)\implies(iii)$}. 
By hypothesis  $\exists\, a,h\in\mathcal{O}_2
\quad \text{such that}
\quad
\eta=a\omega_0+dh.
$
Thus $\eta=xya\frac{\omega_0}{xy}+dh=\tilde a\frac{\omega_0}{xy}+dh$ with $\tilde a \in \mathcal O_2$, proving (iii). 
Assume now that $\eta=\tilde a\frac{\omega_0}{xy}+dh$ for some $\tilde a ,\tilde h \in \mathcal O_2$. 
Given a closed path $\gamma \subset \phi_c$ we have $\int_\gamma \tilde a \frac{\omega_0}{xy} =0$ because $\phi_c$ is invariant by $\omega_0$. Also we have $\int_\gamma d\tilde h =0$ by the Fundamental theorem of Calculus. 
Thus we have $(iii) \implies (i)$. 
\end{proof}

\section{Polynomial case}
\label{section:polyncase}

We shall begin with a simple remark about a global version of Proposition~\ref{proposition:homology}.

\begin{lemma}
\label{lemma:homologyglobal}
Given polynomials $f(z,w),g(z,w)$ and relatively prime numbers
$p,q\in\mathbb{N}$ we put
\[
\omega_0 = p g\,df - q f\,dg.
\]
Let $a(z,w)$ be a given polynomial.
For a given polynomial $h(z,w)$ we finally define
\[
\Omega = a \omega_0 + dh.
\]
Then
\[
\int_{\gamma_c}\Omega = 0
\]
for every closed path $\gamma_c \subset \phi_c
=
\left\{\frac{f(z,w)^p}{g(z,w)^q}=c\right\}
\subset \mathbb{C}^2.
$
\end{lemma}
\begin{proof}
\[
\int_{\gamma_c} dh = 0
\]
as we know from the Fundamental Theorem of Calculus.

Now
\[
\int_{\gamma_c} a \omega_0
=
\int_0^1
a(\gamma(t))\, \omega_0(\gamma(t))\cdot \gamma'(t)\, dt
\]
for any parametrization $\gamma(t)$, $0\le t\le1$ of $\gamma_c\subset\phi_c$.

But if we write $\gamma(t)=(z(t),w(t))$
then since $\gamma\subset\phi_c
=
\left\{\frac{f(z,w)^p}{g(z,w)^q}=c\right\}$
we conclude that
\[
f(\gamma(t))^p = c\, g(\gamma(t))^q,
\quad \forall t,
\]
and therefore
\[
\frac{d}{dt}
\left(
\frac{f(\gamma(t))^p}{g(\gamma(t))^q}
\right)=0.
\]

This implies
\[
p\,\frac{d}{dt}\big(f(\gamma(t))\big)\,g(\gamma(t))
-
q\,f(\gamma(t))\,\frac{d}{dt}\big(g(\gamma(t))\big)
=0,
\]
and then
\[
\omega_0(\gamma(t))\cdot \gamma'(t)=0.
\]

Therefore
\[
\int_{\gamma_c} a \omega_0 = 0.
\]
\end{proof}

Let us investigate the possible converses of Lemma~\ref{lemma:homologyglobal}.
In what follows we consider  a polynomial $1$-form 
$
\Omega(x,y)=A(x,y)dx+B(x,y)dy
$
in the complex plane, polynomials $f(z,w), g(z,w)\in\mathbb{C}[z,w]$ and natural numbers   $p,q\in\mathbb{N}$ with $\langle p,q\rangle=1$.
Let us denote by $\phi$ the rational function
$
\phi(z,w)=\frac{(f(z,w))^p}{(g(z,w))^q}.
$
and by $\omega_0$ the polynomial $1$-form
\[
\omega_0(z,w)=p g df - q f dg.
\]
For any value $c\in\mathbb{C}\cup\{\infty\}$ the curve $\phi_c\subset \mathbb C^2$ is given by
\[
\phi_c=\{\phi(z,w)=c\}.
\]
 The singular set of $\omega_0$ is 
\[
\mathrm{Sing}(\omega_0)
=
\{\xi\in\mathbb{C}^2;\ A(\xi)=B(\xi)=0\}.
\]
and the foliation $\mathcal{F}=\mathcal{F}(\omega_0)$ induced by
$\omega_0$ in the complex projective plane
\[
\mathbb{P}^2=\mathbb{C}^2\cup \ell_\infty,
\]
where $\ell_\infty\cong\mathbb{P}^1$ is
the line at infinity.

 We assume that $\deg A=\deg B=k$. Moreover, we shall assume $(f,g)$ is a Morse type pair, in the sense of Definition~\ref{definition:*generic}, as a consequence we have: 
$\mathrm{sing}(\omega_0)$ is finite (and discrete) so that
\[
\mathrm{sing}(\mathcal{F})
=
\mathrm{sing}(\omega_0)\cup(\mathrm{sing}(\mathcal{F})\cap \ell_\infty)
\]
which is also finite.

\begin{remark}
\begin{enumerate}
\item Our hypotheses allow the curves  $S_f$ and $S_g$
to connect a saddle-singularity to a dicritical
singularity.
\item the dicritical singularities of
$\mathcal{F}=\mathcal{F}(\omega_0)$
are given by the indeterminacy points
of the map
\[
\phi:\mathbb{CP}^2 \longrightarrow \mathbb{C}\cup\{\infty\},
\qquad
\phi=\frac{f^p}{g^q}.
\]
\item the non-dicritical singularities of $\mathcal F$ admit local holomorphic 
first integrals (\cite{mattei-moussu}).
\end{enumerate}
\end{remark}
Let us check our hypothesis above in
a concrete example:

\begin{example}
We depict the global picture of the pencil
\[
\phi=\frac{f^2}{g^3}=\frac{x^2}{y^3}.
\]
on $\mathbb CP^2$ (see figure~\ref{fig:paper-fig1}). The origin of the system $(x,y)$ is a dicritical singularity. 
We consider the systems of coordinates $(u,v)$ and $(r,s)$ given by
\[
x=u^{-1},\quad y=v\,u^{-1}
\quad\Rightarrow\quad
\frac{x^2}{y^3}=c
\iff
\frac{1/u^2}{v^3/u^3}=c
\]

\[
\iff
u=c v^3
\qquad
\text{(dicritical singularity)}.
\]

\[
r=\frac{1}{y},\quad s=\frac{x}{y}
\quad\Rightarrow\quad
\frac{x^2}{y^3}=c
\iff
\frac{s^2/r^2}{1/r^3}=c
\]

\[
\iff
 s^2r=c
\qquad
\text{(non-dicritical saddle)}.
\]

\begin{figure}[htbp]
\centering
\includegraphics[width=0.4\linewidth]{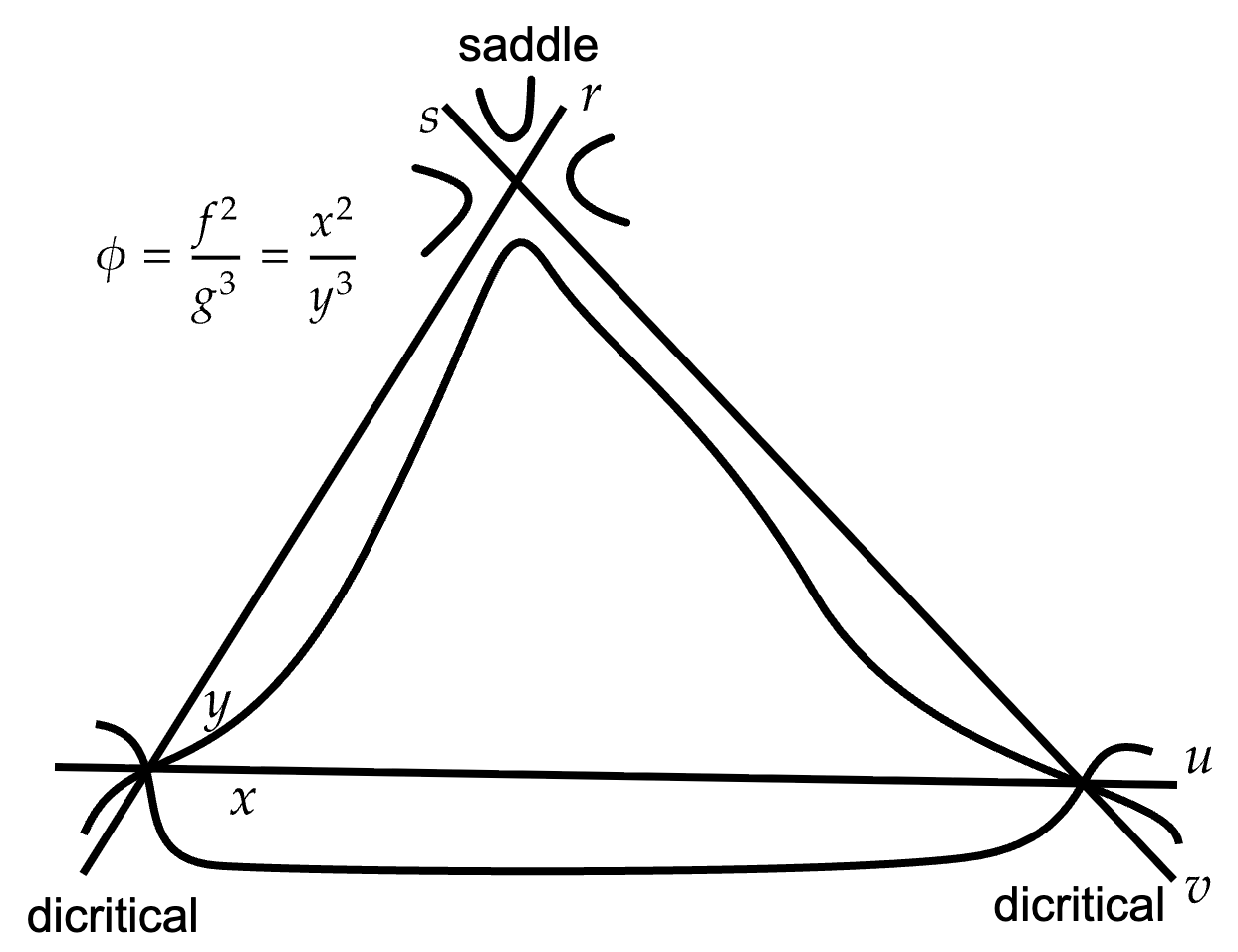}
\caption{Global picture of $x^2 / y^3 =c$ on $\mathbb CP^2$.}
\label{fig:paper-fig1}
\end{figure}

\end{example}

Now we go back to the situation
in the beginning. We also consider the
following homological hypothesis:

Assume that
\begin{equation}
\int_\gamma \frac{\Omega}{f g}=0
\label{equation:cohofg}
\end{equation}
for every closed path $\gamma\subset\phi_c$,
for every $c$.

\begin{proof}[Proof of Theorem~\ref{theorem:firsintegration}]

We assume that the curve
\[
S_f:=\overline{\{f(z,w)=0\}}\subset\mathbb{CP}^2
\]
exhibits a single dicritical singularity,
say $\xi_0\in S_f\cap S_g$
and $\xi_0\in\mathbb{C}^2$ (see figure~\ref{fig:paper-fig2}).

We choose local coordinates $(x,y)\in U\subset\mathbb{C}^2$
centered at $\xi_0$ such that


$f(x,y)=x$, $g(x,y)=y$ and therefore
\[
\phi(x,y)=\frac{x^p}{y^q}.
\]
We have
\[
\omega_0 = p y\,dx - q x\,dy.
\]

From the hypothesis in the homology \eqref{equation:cohofg}
\[
\int_{\gamma_0}\frac{\Omega}{f g}=0
\ \ \text{if }\gamma_0 \text{ is any closed path}
\]
contained in $U\cap\phi_c$, $\forall c\in\mathbb{C}\cup\{\infty\}$.

\begin{figure}[htbp]
\centering
\includegraphics[width=0.25\linewidth]{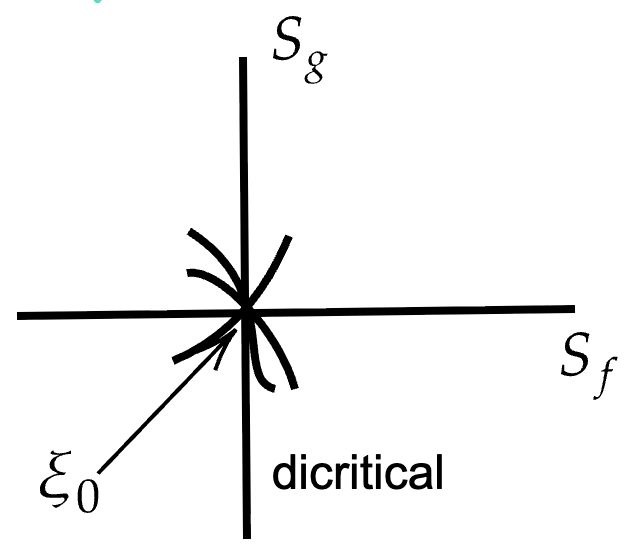}
\caption{The dicritical singularity $\xi_0 \in S_f \cap S_g$.}
\label{fig:paper-fig2}
\end{figure}


This implies by Proposition~\ref{proposition:homology} that for
\[
\eta:=\frac{\Omega}{f g}=\frac{\Omega(x,y)}{x y}
\]
we have
\[
\eta\wedge \omega_0 = dh_U\wedge \omega_0
\]
or, equivalently,
\begin{equation}
\label{equation:hU}
\eta\wedge \frac{\omega_0}{fg} = dh_U\wedge \frac{\omega_0}{fg}
\end{equation}
for some holomorphic function $h_U:U\to\mathbb{C}$.

Now we aim to extend the function
\[
h_U:U\to\mathbb{C}
\]
to the whole projective plane.

This will be done in two steps:

First step: We can extend $h_U$ to a
neighborhood of $S_f\setminus(S_f\cap\mathrm{sing}(\mathcal{F}))=:S_f^{*}$.

We fix a point $Q\in S_f\cap U$, $Q\neq \xi_0$ and
a transverse disk $\Sigma_Q$ centered at $Q$, $\Sigma_Q\subset U$.

Given any point $Q_1\in S_f^{*}$ we choose a
transverse disk $\Sigma_{Q_1}$ centered at $Q_1$ (see figure~\ref{fig:paper-fig3}).
\begin{figure}[htbp]
\centering
\includegraphics[width=0.35\linewidth]{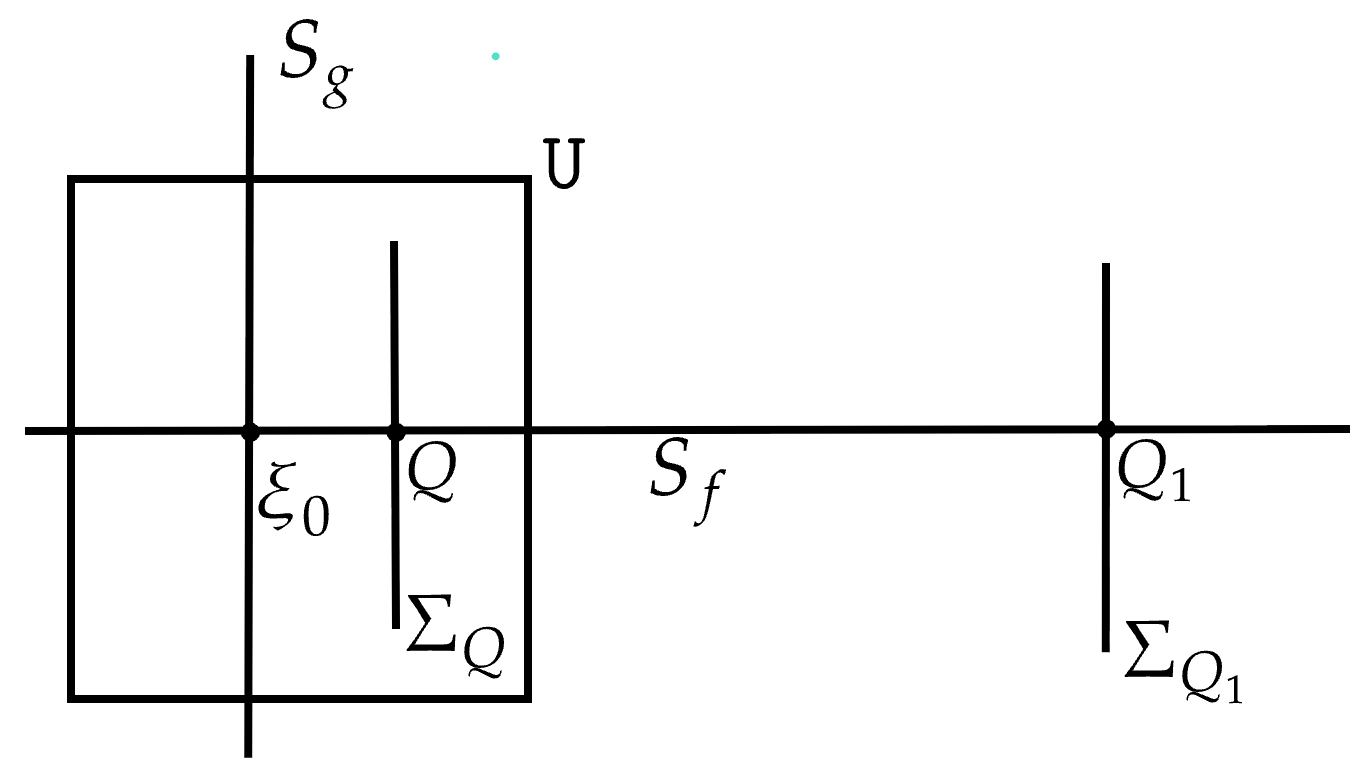}
\caption{Preparation for the construction of the extension.}
\label{fig:paper-fig3}
\end{figure}

Take any path
\[
\delta_ Q^{Q_1}:[0,1]\longrightarrow S_f^{*}
\]
such that $\delta_ Q^{Q_1}(0)=Q$, $\delta_ Q^{Q_1}(1)=Q_1$ (cf. figure~\ref{fig:paper-fig4}).

\begin{figure}[htbp]
\centering
\includegraphics[width=0.4\linewidth]{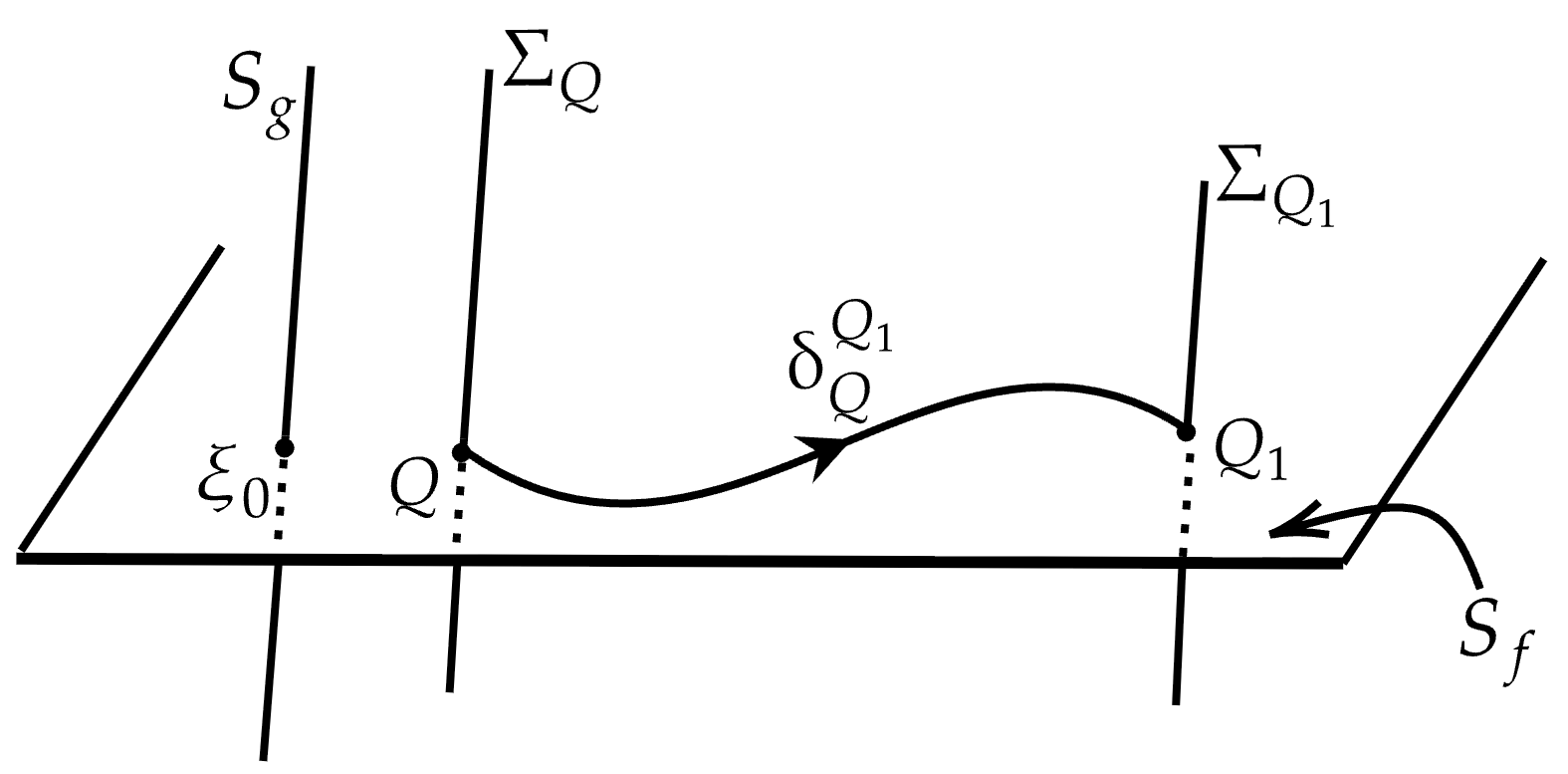}
\label{fig:paper-fig4}
\caption{A path $\delta_Q^{Q_1}$ in $S_f^{*}$ joining $Q$ to $Q_1$, with transverse disks.}
\end{figure}

We then define
\[
h(Q_1):=\int_{\delta_ Q^{Q_1}}\Omega + h_U(Q).
\]

Notice that $\Omega$ is holomorphic so that
its restriction to each curve $\phi_c$ is closed.

Since $\int_\gamma\Omega=0$, $\forall$ closed path
\[
\gamma\subset S_f^{*}
\]
it follows that $h(Q_1)$ is well-defined and does not depend on
the chosen path $\delta_ Q^{Q_1}$ joining $Q$ to $Q_1$
in $S_f^{*}$.

If we take a point $\tilde{Q}_1\in\Sigma_{Q_1}$ close
enough to $Q_1$ then the curve $\phi$ that
contains $\tilde{Q}_1$ (put $c=\phi(\tilde{Q}_1)$)
will also intersect $\Sigma_Q$ at a point $\tilde{Q}$
close to $Q$ (figure~\ref{fig:paper-fig5}).
\begin{figure}[htbp]
\centering
\includegraphics[width=0.4\linewidth]{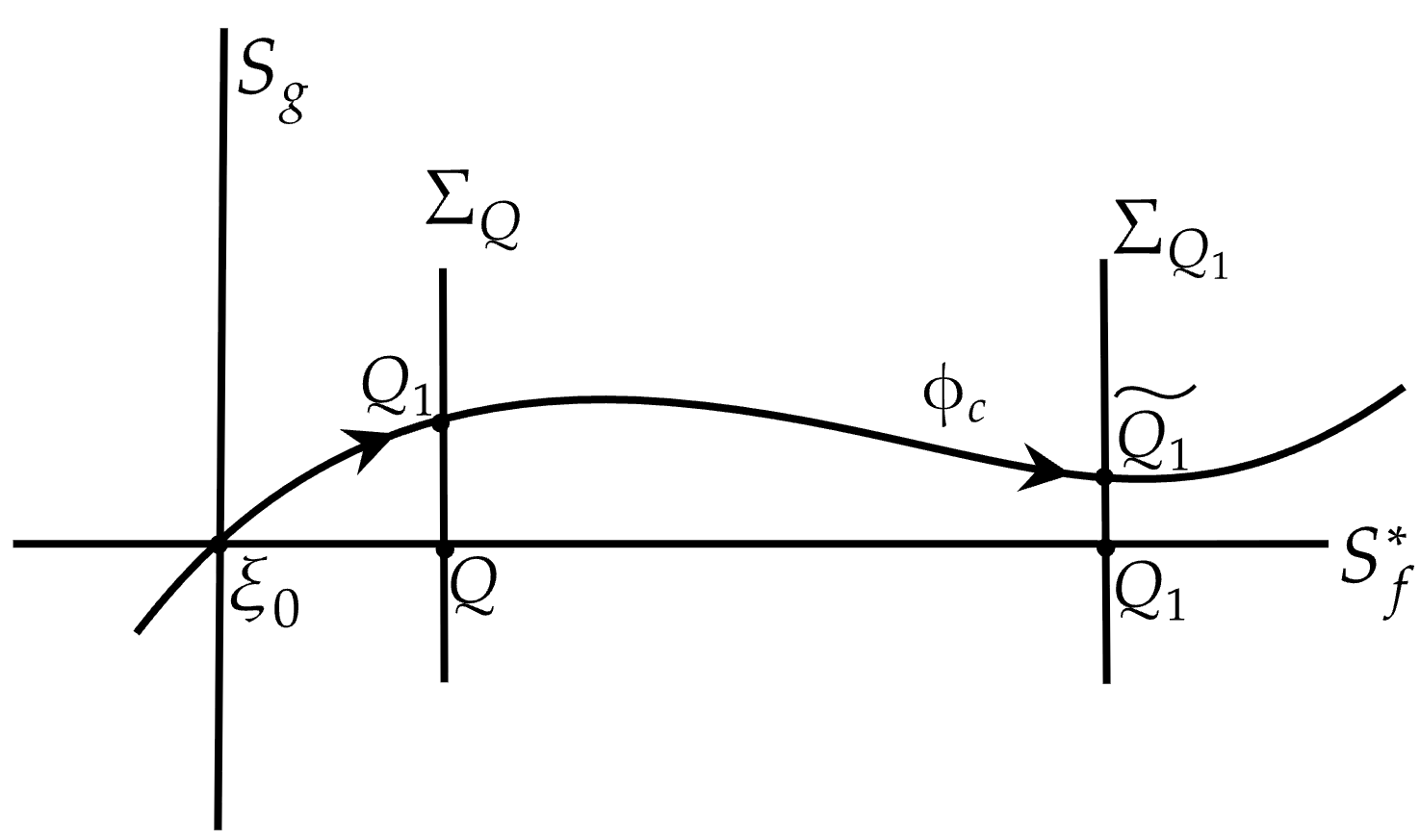}
\caption{A nearby fiber $\phi_c$ intersecting both transverse disks at $\tilde Q$ and $\tilde Q_1$.}
\label{fig:paper-fig5}
\end{figure}

Using these arguments we extend
$h$ to a neighborhood of $S_f^{*}$,
the regular point of $S_f$, together with $U\ni \xi_0$.

\begin{figure}[htbp]
\centering
\includegraphics[width=0.45\linewidth]{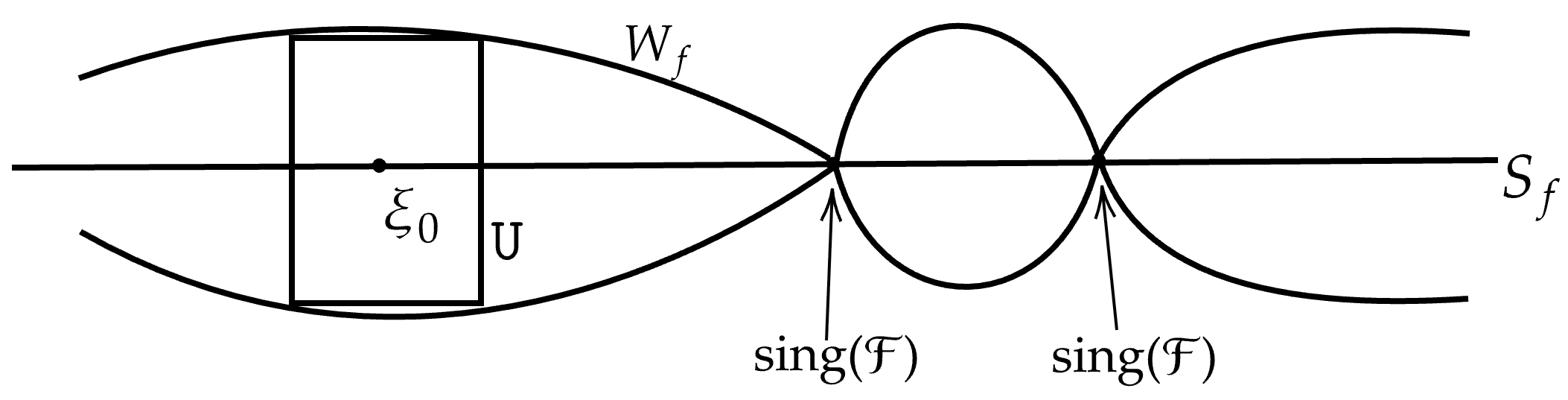}
\caption{Neighborhood $W_f$ of $S_f^{*}$ together with $U\ni \xi_0$.}
\label{fig:paper-fig6}
\end{figure}
We shall denote by $W_f$ such neighborhood (figure~\ref{fig:paper-fig6}).
Take another singular point $P_1\in S_f\cap\mathrm{sing}(\mathcal{F})$.
By the genericity hypothesis of $f(z,w)$ and $g(z,w)$
there are two options:

{\bf 1st}. $P_1$ is non-dicritical. In this case, because of the Morse type pair hypothesis of $f$ and $g$, 
the singularity is a linearizable saddle.
From the saddle-linearizable case (cf. \cite{scarduaIlyshenko}) 
 we know that the restrictions of
$h$ to a disk transverse to $S_f^{*}$ and
centered at a point $Q_1\approx P_1$, $Q_1\neq P_1$
extends to a neighborhood of
the singularity of $P_1$ satisfying still the
cohomological equation.

This shows that $h$ extends to

Then we define
\[
h(\tilde Q_1):=h(\tilde Q)+\int_{\delta_{\tilde Q}^{\tilde Q_1}}\Omega
\]
for any path
\[
\delta_{\tilde Q}^{\tilde Q_1}:[0,1]\to \phi_c
\]
joining $\delta_{\tilde Q}^{\tilde Q_1}(0)=\tilde Q$
to $\delta_{\tilde Q}^{\tilde Q_1}(1)=\tilde Q_1$.

\begin{figure}[htbp]
\centering
\includegraphics[width=0.45\linewidth]{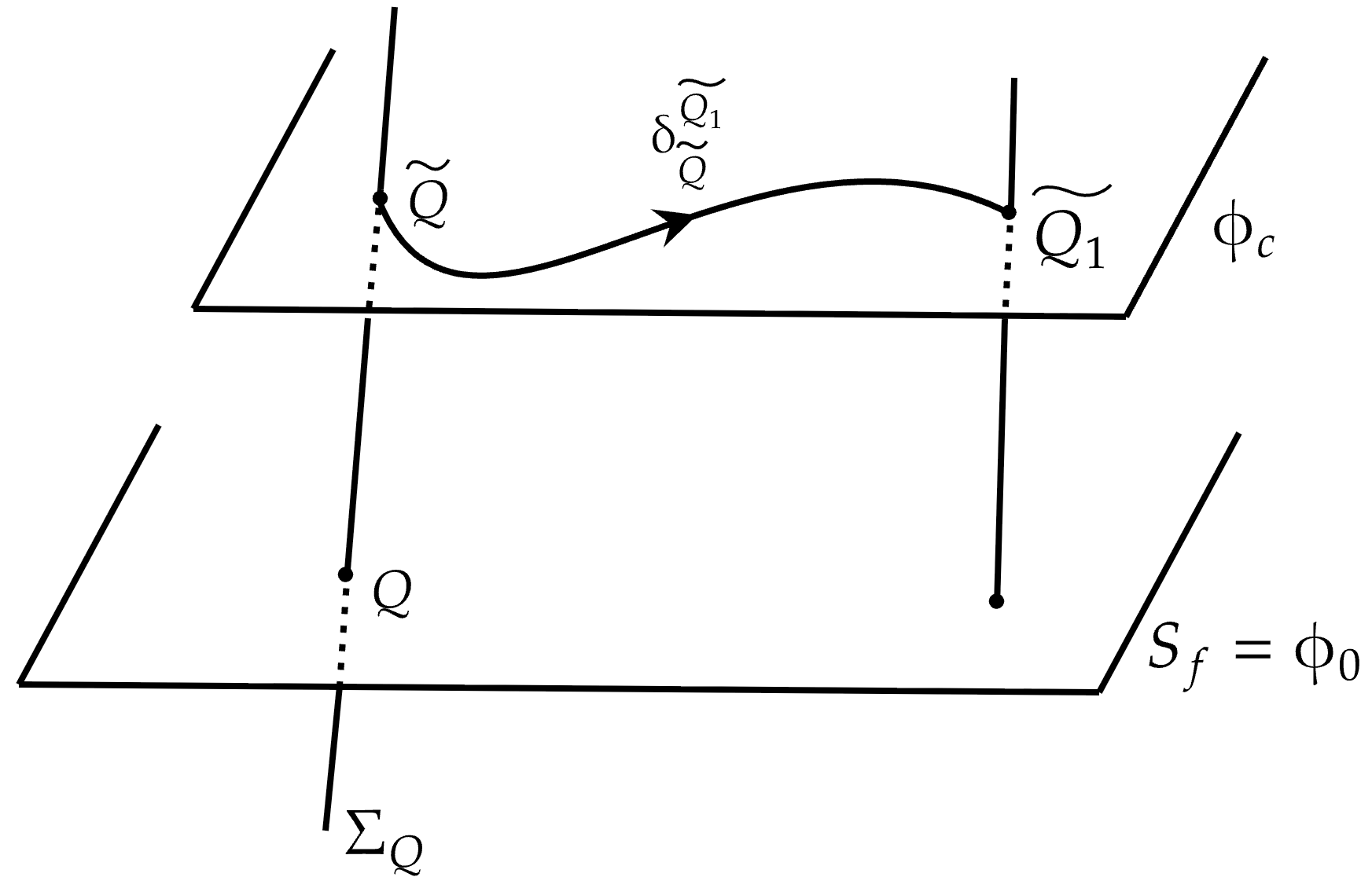}
\label{fig:paper-fig7sete}
\caption{Definition of $h(\tilde Q_1)$ using a path in the fiber $\phi_c$.}
\end{figure}

Observe that $h$ and $h_U$ coincide
in $U$ whenever $h$ is also defined, 
a neighborhood of $P_1$.

{\bf 2nd}. $P_1$ is a dicritical singularity.

In this case $S_f$ connects two dicritical
singularities $(\xi_0 \text{ and } P_1)$.
This case is excluded by hypothesis.
This ends the proof of the existence of
$h$ in a neighborhood $W$ of $S_f$ on $\mathbb{CP}^2$.

By Levi extension theorem  \cite{[C-LN-Sannals]},\cite{[Si]}
this function $h$ extends to $\mathbb{CP}^2$
and it is therefore rational.
By construction $h$ is holomorphic in the affine
part of $S_f$ and therefore $h$ is polynomial.

We have now obtained a polynomial
$h(z,w)\in\mathbb{C}[z,w]$ such that

\[
\frac{\Omega}{f g}\wedge \omega_0
=
dh \wedge \omega_0 .
\]
i.e.,
\[
\eta\wedge \frac{\omega_0}{fg} = dh_U\wedge \frac{\omega_0}{fg}
\]

The theorem follows from Saito-de Rham
lemma \cite{[Saito]}.
\end{proof}

A very similar proof, also based in Proposition~\ref{proposition:homology}, can be given for the following:

\begin{theorem}[second integration lemma]
\label{theorem:secondinteglemma}
Let $f(z,w), g(z,w)$ be homogeneous polynomials
$f,g\in\mathbb{C}[z,w]$ with no common factor.
Suppose that $f,g$ are in general position
and satisfy the Morse type pair condition.

Let
\[
\Omega = A(z,w)\,dz + B(z,w)\,dw
\]
be a polynomial $1$-form in $\mathbb{C}^2$
and assume that
\begin{equation}
\label{equation:homology}
\int_{\gamma_c}\Omega = 0
\quad \text{for every closed path } \gamma_c \subset \phi_c
\quad \forall c\in\mathbb{C}\cup\{\infty\}.
\end{equation}
where where
\[
\phi=f^p/g^q \in \mathbb C(z,w)
\]
Then there are polynomials
\[
a(z,w),\, h(z,w)\in\mathbb{C}[z,w]
\]
such that
\[
\Omega = a \omega_0 + dh
\]
for
\[
\omega_0 = p g\,df - q f\,dg.
\]

If moreover we have
\[
\deg \omega_0 = \deg \Omega
\]
then we may take $a=a_0\in\mathbb{C}$
and
\[
\deg h \le \deg \omega_0 + 1.
\]
\end{theorem}

The proof is a consequence of the proof given for Theorem~\ref{theorem:firsintegration} and of the following two lemmas~\ref{lemma:normalformpolynomial} and ~\ref{lemma:aisconstant}.
\begin{lemma}
\label{lemma:normalformpolynomial}
Let $f,g\in\mathbb{C}[z,w]$ be homogeneous
polynomials without common factors.
Let also $\Omega$ be a polynomial $1$-form
\[
\Omega = A(z,w)\,dz + B(z,w)\,dw,
\quad A,B\in\mathbb{C}[z,w],
\]
and suppose that
\[
\Omega = a \omega_0 + dh
\]
for some polynomials $a,h\in\mathbb{C}[z,w]$.

Suppose that
\[
\deg\Omega=\deg \omega_0=\deg f+\deg g-1
\]
for
\[
\omega_0 = p g\,df - q f\,dg.
\]

Then
\[
\Omega = a_0 \omega_0 + d\tilde h
\]
for some $a_0\in\mathbb{C}$
and polynomial $\tilde h$ of degree $\le \nu+1$.
\end{lemma}

\begin{proof}[Proof of the lemma.]
We write
\[
h=\sum_{i\ge0} h_i
\]
as a sum of homogeneous polynomials $h_i$ of degree $i\ge0$.

Similarly, we write
\[
a=\sum_{j\ge0} a_j
\]
with $a_j$ homogeneous of degree $j\ge0$.

By hypothesis $f,g$ are homogeneous
so that $\omega_0$ is homogeneous of degree
\[
\nu=\deg f+\deg g-1.
\]
Moreover, we also assume that $\Omega$ has degree $\nu$.

From $\Omega=a\omega_0+dh$ we have
\[
\Omega
=
\sum_{j\ge0} a_j \omega_0
+
\sum_{i\ge0} d h_i .
\]

Since $\omega_0$ is homogeneous of degree $\nu$,
$a_j\omega_0$ is homogeneous of degree $j+\nu$.

Therefore we can write
\[
\Omega
=
a_0 \omega_0
+
\sum_{j\ge1} a_j \omega_0
+
\sum_{i\le \nu+1} d h_i
+
\sum_{i\ge \nu+2} d h_i .
\]

Let
\[
\xi
:=
\sum_{j\ge1} a_j \omega_0
+
\sum_{i\ge \nu+2} d h_i .
\]

Then $\xi$ is a sum of homogeneous forms,
each of degree $\ge \nu+1$.

On the other hand,
\[
\Theta
:=
a_0 \omega_0
+
\sum_{i\le \nu+1} d h_i
\]
is a sum of homogeneous forms,
each of degree $\le \nu$.

Since $\Omega$ is a form of degree $\nu$
it follows that
\[
\Omega=\Theta.
\]

Hence
\[
\Omega
=
a_0 \omega_0
+
d\!\left(\sum_{i\le \nu+1} h_i\right)
=
a_0 \omega_0 + d\tilde h,
\]
for
\[
\tilde h := \sum_{i\le \nu+1} h_i.
\]
\end{proof}
\begin{lemma}
\label{lemma:aisconstant}
Let $f,g\in\mathbb{C}[z,w]$ be polynomials without common factors.
Let also $\Omega$ be a polynomial $1$-form
\[
\Omega = A(z,w)\,dz + B(z,w)\,dw,
\quad A,B\in\mathbb{C}[z,w],
\]
and suppose that
\[
\Omega = a \omega_0
\]
for some polynomial $a\in\mathbb{C}[z,w]$.
Suppose that $\deg\Omega = \deg \omega_0$.
Then $a\in\mathbb{C}$ is a constant.
\end{lemma}
\begin{proof} Immediate.
\end{proof}

\section{Cusps}
We shall now apply our techniques to the study of polynomial 1-forms on $\mathbb C^2$, which define 
a cusp singularity at the origin $0\in \mathbb C^2$. 
\label{section:cusps}
\subsection{Cusp singularities}

We recall that a singularity
\[
\Omega(x,y)=a(x,y)\,dx + b(x,y)\,dy
\in \Omega^1(\mathbb{C}^2,0)
\]
at the origin $0\in\mathbb{C}^2$
is called a {\it cusp type singularity}
if it has a nilpotent first jet of the form
\[
j^1\Omega = y^{k-1}dy,
\]
and exhibits a single separatrix
necessarily of the form
\[
y^k + x^\ell = 0,
\]
and is a generalized curve in the sense
established by \cite{[C-LN-S]}.

In this case, according to \cite{Takens},  we may choose formal coordinates and rewrite $\Omega$ as

\[
\Omega
=
d(y^k + x^\ell)
+
B(x,y)(\ell y\,dx - kx\,dy)
\]

for some formal holomorphic $B(x,y)\in\mathcal{O}_2$. We shall say that $\Omega$ is a 
{\it cusp form} if this formal normal form is in fact convergent. This is the case for 
$\ell=2, k=3$ for instance (\cite{Cerveau-Moussu}). More details can be found in \cite{Takens},\cite{Cerveau-Moussu} and \cite{CerveauLoray1998} (appendix).

\begin{definition}
\label{definition:polynomialcuspform}
Assume now that $\Omega$ is polynomial.
We shall say that $\Omega$ {\it is a polynomial
cusp form} if
$
\Omega
=
d(f^p + g^q)
+
B(p g\,df - q f\,dg)
$ for some polynomials $f,g,B$ where
$p,q\in\mathbb{N}$,
$\langle p,q\rangle=1$. This is called a {\it constant coefficient } cusp form if $B$ is a constant.
\end{definition}

\subsection{Examples}
We study a couple of important examples on the zoology of polynomial cusp forms.
We start by showing that  {\em not all cusp forms $\Omega$
satisfying the homological conditions
\[
\int_{\gamma}\Omega=0,
\quad \forall\gamma\subset\phi_c,\ \forall c,
\]
admit a meromorphic first integral.}

\begin{example}
Given polynomials $f(z,w),g(z,w)$ and relatively prime numbers
$p,q\in\mathbb{N}$ we put
\[
\omega_0 = p g\,df - q f\,dg.
\]
Let $a(z,w)$ be a given polynomial.
For a given polynomial $h(z,w)$ we finally define
\[
\Omega = a \omega_0 + dh.
\]

\begin{claim}
\[
\int_{\gamma_c}\Omega = 0
\]
for every closed path $\gamma_c \subset \phi_c
=
\left\{\frac{f(z,w)^p}{g(z,w)^q}=c\right\}
\subset \mathbb{C}^2.
$
\end{claim}
\begin{proof}
\[
\int_{\gamma_c} dh = 0
\]
as we know from the Fundamental Theorem of Calculus.

Now
\[
\int_{\gamma_c} a \omega_0
=
\int_0^1
a(\gamma(t))\, \omega_0(\gamma(t))\cdot \gamma'(t)\, dt
\]
for any parametrization $\gamma(t)$, $0\le t\le1$ of $\gamma_c\subset\phi_c$.

But if we write $\gamma(t)=(z(t),w(t))$
then since $\gamma\subset\phi_c
=
\left\{\frac{f(z,w)^p}{g(z,w)^q}=c\right\}$
we conclude that
\[
f(\gamma(t))^p = c\, g(\gamma(t))^q,
\quad \forall t,
\]
and therefore
\[
\frac{d}{dt}
\left(
\frac{f(\gamma(t))^p}{g(\gamma(t))^q}
\right)=0.
\]

This implies
\[
p\,\frac{d}{dt}\big(f(\gamma(t))\big)\,g(\gamma(t))
-
q\,f(\gamma(t))\,\frac{d}{dt}\big(g(\gamma(t))\big)
=0,
\]
and then
\[
\omega_0(\gamma(t))\cdot \gamma'(t)=0.
\]

Therefore
\[
\int_{\gamma_c} a \omega_0 = 0.
\]

This shows the claim.
\end{proof}

Let us  proceed with a more complete example.

Let
\[
\Omega = a(z,w)(2w\,dz-3z\,dw) + d(z^2+w^3),
\quad a(z,w)\in\mathbb{C}[z,w].
\]

Then
\[
\Omega = d(z^2+w^3) + a(z,w)(2w\,dz-3z\,dw)
\]
describes a cusp type singularity
and for $a(z,w)\neq0$ we have a single separatrix
given by
\[
z^2+w^3=0.
\]
Its reduction of singularities produces
a $3$ components exceptional divisor.
The central divisor has a holonomy group generated by two maps

\[
\gamma_1,\gamma_2\in \mathrm{Diff}(\mathbb{C},0)
\quad\text{s.t.}\quad
\gamma_1^2=\gamma_2^3=\mathrm{Id}.
\]

By suitable choices of $a(z,w)$ we may obtain holonomy
groups
\[
H<\mathrm{Diff}(\mathbb{C},0),
\]
\[
H=\langle\gamma_1,\gamma_2\rangle
\]
with $H$ non-solvable
and, in particular, $H$ is not finite,
so that the corresponding foliation
\[
\Omega=0
\]
does not admit a meromorphic
or holomorphic first integral (see \cite{mattei-moussu}).

This shows that {\em not all cusp  $1$-forms $\Omega$
satisfying the homological conditions
\[
\int_{\gamma}\Omega=0,
\quad \forall\gamma\subset\phi_c,\ \forall c,
\]
admit a meromorphic first integral.}
\end{example}
Our next example shows that a polynomial cusp form with constant coefficients exhibits a liouvillian first integral, of hypergeometrical type: 
\begin{example}[hypergeometrical first integral]
\label{example:hypergeometricalcusp}
Let us give a more explicit example.
We take the $1$-form
\[
\Omega = 2y\,dx - 3x\,dy + d(x^2+y^3).
\]

Let us put
\[
\psi = x^2+y^3,
\quad
\phi = \frac{x^2}{y^3}.
\]

Then from $(x,y)$ to $(\phi,\psi)$
we have the following ramified
change of coordinates:

\[
\left\{
\begin{aligned}
x &= \pm \sqrt{\frac{\phi\psi}{\phi+1}},\\
y &= \left(\frac{\psi}{\phi+1}\right)^{1/3}.
\end{aligned}
\right.
\]

From
\[
\Omega = (2y\,dx - 3x\,dy) + d(x^2+y^3)
\]
we get
\[
\Omega
=
\frac{d(x^2/y^3)}{x^2/y^3}
+
d(x^2+y^3).
\]

Hence
\[
\Omega = \frac{d\phi}{\phi} + d\psi.
\]

Substituting the change of coordinates,

\[
\Omega
=
\pm
\left(\frac{\phi\psi}{\phi+1}\right)^{1/2}
\left(\frac{\psi}{\phi+1}\right)^{1/3}
\frac{d\phi}{\phi}
+
d\psi
\]

\[
=
\pm
\frac{\phi^{1/2}\psi^{5/6}}{(\phi+1)^{5/6}}
\frac{d\phi}{\phi}
+
d\psi
\]

\[
=
\psi^{5/6}
\left[
\pm
\frac{\phi^{-1/2}}{(\phi+1)^{5/6}}\,d\phi
+
\frac{d\psi}{\psi^{5/6}}
\right].
\]

The given $1$-form admits the Liouvillian
first integral

\[
\mathfrak{F}=\int \frac{d\psi}{\psi^{5/6}}
\;+\;
\int
\frac{d\phi}{\phi^{1/2}(\phi+1)^{5/6}}.
\]

Let us take a closer look at the
second integral.

For the classical substitution
\[
t=\frac{\phi}{1+\phi}
\]
we obtain
\[
\phi=\frac{t}{1-t},
\qquad
d\phi=\frac{dt}{(1-t)^2}.
\]

Then
\[
\int
\frac{d\phi}{\phi^{1/2}(1+\phi)^{5/6}}
=
B_t\!\left(\frac12,\frac13\right)
+
C,
\]
where $B_t(\alpha,\beta)$ is the incomplete
Beta function or, in terms of the hypergeometric function,
we have

\[
\int
\frac{d\phi}{\phi^{1/2}(1+\phi)^{5/6}}
=
2\sqrt{\frac{\phi}{1+\phi}}\;
{}_2F_1
\!\left(
\frac12,\frac23;\frac32;\frac{\phi}{1+\phi}
\right)
+
C.
\]
Thus $\mathcal{F}(\Omega)$ exhibits the liouvillian first integral (\cite{C-S ETDS})
\[
f
= 6\psi^{1/6}
+ B_t\!\left(\frac12,\frac13\right)+C,
\]
where
\[
\psi=x^2+y^3,
\qquad
t=\frac{\phi}{1+\phi}
=\frac{x^2/y^3}{1+x^2/y^3}
=\frac{x^2}{x^2+y^3}.
\]

Therefore
\[
\mathfrak {F}
= 6(x^2+y^3)^{1/6}
+ B_{\frac{x^2}{x^2+y^3}}\!\left(\frac12,\frac13\right)+C.
\]

In terms of the hypergeometric function we have
\[
I
= \int \frac{d\phi}{\phi^{1/2}(1+\phi)^{5/6}}
= 2\sqrt{\frac{\phi}{1+\phi}}\;
{}_2F_1\!\left(\frac12,\frac23;\frac32;\frac{\phi}{1+\phi}\right).
\]

For the hypergeometric function
\[
{}_2F_1(a,b;c;z)
:= \sum_{m=0}^{\infty}
\frac{(a)_m(b)_m}{(c)_m}\,z^m,
\quad |z|<1.
\]

In these terms the first integral $\mathfrak {F}$ is given by
\[
\mathfrak {F}
= 6(x^2+y^3)^{1/6}
+ 2\sqrt{\frac{x^2/y^3}{1+x^2/y^3}}\;
{}_2F_1\!\left(\frac12,\frac23;\frac32;\frac{x^2}{x^2+y^3}\right).
\]

\bigskip

\end{example}

\begin{remark}{\rm 
The classical incomplete Beta function is defined as

\[
B_x(\alpha,\beta)
=
\int_0^x
t^{\alpha-1}(1-t)^{\beta-1}\,dt,
\]
while the hypergeometric function is given by

\[
{}_2F_1(a,b;c;z)
=
\sum_{m=0}^{\infty}
\frac{(a)_m(b)_m}{(c)_m m!}\,
z^m,
\]

for $|z|<1$, which is the fundamental solution
of the second order linear ODE below

\[
z(1-z)u''
+
[c-(a+b+1)z]u'
-
abu
=
0,
\]

called Gauss hypergeometric equation.

Recall that $(\alpha)_m$ stands for the rising factorial or Pochhammer symbol
\[
(\alpha)_m = \alpha(\alpha+1)\cdots(\alpha+m-1),
\]
so that ${}_2F_1(a,b;c;z)$ is defined for complex numbers $a,b,c$ and $c\notin\mathbb Z_0^{-}$.

}
\end{remark}

\subsection{Globalization theorems for cusp forms}

Let us now prove the existence of a polynomial cusp form under 
certain local cusp condition and the vanishing of the form in the 
first homology of a suitable pencil.

\begin{proof}[Proof of Theorem~\ref{theorem:firstglobalizationcusp}]

Let $\Omega$ be given as in the statement. 
From the hypothesis
\[
\int_{\gamma_c} \Omega = 0,
\qquad
\forall\, \gamma_c \subset \phi_c,
\qquad
\phi_c=\left\{\frac{f(z,w)^2}{g(z,w)^3}=c\right\},
\]
$c\in\mathbb{C}\cup\{\infty\}$,
we obtain
from Theorem~\ref{theorem:secondinteglemma}  
\[
\Omega = a\omega_0 + dh
\]
for some $a(z,w), h(z,w)\in \mathbb C[z,w]$. We may assume that $h(0,0)=0$.
Since the cusp $f^2 + g^3$ is invariant we conclude that $(f^2 + g^3) \big|(f^2 + g^3)\wedge dh(z,w)$ in $\mathbb C[z,w]$ and 
therefore $(f^2 + g^3) \big|h(z,w)$ so that we have $h=(f^2 + g^3) h_1$ for some $h_1 \in \mathbb C[z,w]$. 
 This gives 
 \[ 
 \Omega= a(3fdg - 2g df ) + d(h_1(f^2 + g^3))
 \]
Since $\deg \Omega \leq \deg (f^2 + g^3)$ we conclude that 
$h_1 \in \mathbb C$ is constant. This shows that $\Omega$ is of the form 
\[ 
\Omega = \lambda [d(f^2 + g^3) + \mu (3fdg - 2g df)]
\]
for some $\lambda, \mu \in \mathbb C$.   
This already proves the first part of the theorem. 
As for the second part, put $\Omega_0:=d(x^2 + y^3) + \mu(3xdy - 2 y dx)$. Then 
$\Omega=\tau^*(\Omega_0)$ for the polynomial map 
$\tau:=(f,g)$. As we have seen $\Omega_0$ admits a first integral of liouvillian hypergeometric type and therefore 
the same will hold for $\Omega$.
\end{proof}
Now we prove that a  polynomial 1-form leaving invariant a cusp $f^p + g ^q=0$ and vanishing in the 1st homology of $f^p/g^q$, having a controled degree, admits a liouvillian 
first integral of hypergeometrical type:
\begin{proof}[Proof of Theorem~\ref{theorem:secondglobalizationcusp}]
For a matter of simplicity we shall consider $p=2$ and $q=3$.
Since $(f,g)$ are generic we may assume that the restrictions $(f\big|_U, g\big|_U)$ to some 
open neighborhood $U$ of the origin $0 \in \mathbb C^2$ are a coordinate system. 
Then, the local hypothesis at the origin implies that 
$\Omega = d(f^2+g^3) + b(2g\,df - 3f\,dg)$ for some $b \in \mathcal O_2$.
On the other hand, by the triviality of $\Omega$ in the first homology groups 
of the fibers $\phi_c$ implies that 
\[
\Omega= a(2g\,df - 3f\,dg) + dh= a \omega_0 + dh
\]
for some $a, h \in \mathbb C[z,w]$.  

Comparing these two writings we obtain
\[
a(2g\,df - 3f\,dg) + dh
=
d(f^2+g^3) + b(2g\,df - 3f\,dg).
\]

Hence
\[
(a-b)(2g\,df - 3f\,dg)
=
d\big[(f^2+g^3) - h\big].
\]

Taking exterior derivatives we obtain
\[
d\!\left[(a-b)fg\right]
\wedge
d\!\left(\frac{f^2}{g^3}\right)
=0.
\]

Thus
\[
d[(a-b)fg] \wedge d(f^2/g^3)=0,
\]
which implies (cf. Stein factorization theorem \cite{Gr-Re})
\[
(a-b)fg = \gamma\!\left(\frac{f^2}{g^3}\right)
\]
for some one-variable function $\gamma$.

Since $f,g$ are polynomial functions in $\mathbb{C}^2$
we conclude that either
\[
\frac{f^2}{g^3}=\lambda \in \mathbb{C}
\quad\text{is constant}
\]
or $\gamma$ is constant.

But $f^2=\lambda g^3$ implies $f$ and $g$ have common
factors, contradiction.

Thus we must have
\[
\gamma=\gamma_0=\text{constant}
\]
and therefore
\[
(a-b)fg=\gamma_0,
\]
which implies $\gamma_0=0$ and hence
\[
a=b.
\]

From
\[
\Omega = a\omega_0 + dh
=
d(f^2+g^3) + b\omega_0
\]
we obtain
\[
dh = d(f^2+g^3),
\]
and therefore we may assume that
\[
h = f^2 + g^3.
\]

This gives us a global writing 
\[
\Omega = a\omega_0 + d(f^2+g^3)
\]
as stated in the first part. The degree restriction on $\deg \Omega$ implies that 
$a$ is a constant (cf. Lemma~\ref{lemma:aisconstant}).  The existence of a liouvillian hypergeometrical type first integral is then a consequence of Example~\ref{example:hypergeometricalcusp}.  

\end{proof}


\section{The (local) case of dimension \(m \ge 2\).}

We shall now investigate pencils \(f^p/g^q = c\) in dimension \(m \ge 2\).
We put
\[
\Omega_0 = p g df - q f dg
\]
where \(f,g \in \mathcal O_m\) are holomorphic germs in general position at \(0 \in \mathbb C^m\).
We assume that \(f(0)=g(0)=0\).

\begin{proposition}
\label{proposition:cohomologyeqdimm}
Let \(\eta\) be a meromorphic 1-form in a neighborhood of the origin
\(0 \in \mathbb C^m\), having simple poles and poles set
\[
(\eta)_\infty \subset \{ f g = 0 \}.
\]

Suppose that
\[
d\eta \wedge \Omega_0 = 0.
\]

Then we have the following cases:

(i) \(p>1, q>1 \Rightarrow
\eta = a \frac{\Omega_0}{fg}
+ \lambda \frac{df}{f}
+ \mu \frac{dg}{g}
+ dh.
\)

(ii) \(p=1<q \Rightarrow
\eta = a \frac{\Omega_0}{fg}
+ \lambda \frac{df}{f}
+ \mu \frac{dg}{g}
+ c \frac{g^{q-1}}{f} dg
+ dh.
\)

(iii) \(q=1<p \Rightarrow
\eta = a \frac{\Omega_0}{fg}
+ \lambda \frac{df}{f}
+ \mu \frac{dg}{g}
+ \tilde c \frac{f^{p-1}}{g} df
+ dh.
\)

(iv) \(p=q=1 \Rightarrow
\eta = a \frac{\Omega_0}{fg}
+ \lambda \frac{df}{f}
+ \mu \frac{dg}{g}
+ \alpha \frac{df}{g}
+ \beta \frac{dg}{f}
+ dh.
\)
\end{proposition}
\begin{proof}

We consider a holomorphic embedding
\[
\delta_0 : (\mathbb C^2,0) \to (\mathbb C^m,0)
\]
in general position with respect to \(\Omega_0, f, g\).
We put \(\eta^* = \delta_0^* \eta\) for the pull-back of \(\eta\) by \(\delta_0\).
Let also
\[
f^* = f \circ \delta_0,
\quad
g^* = g \circ \delta_0.
\]

Then, according to the proposition in dimension two, we have the following cases.

1st: \(p>1, q>1\):

\[
\eta^* = a^* \frac{\Omega_0^*}{f^* g^*}
+ \lambda \frac{df^*}{f^*}
+ \mu \frac{dg^*}{g^*}
+ dh^*.
\]

Since \(d\eta \wedge \Omega_0 = 0\), by the extension theorem
\cite{[Cerveau--Berthier]}(Extension theorem p.405, considering the hypothesis $d \eta \wedge \omega_0=0$) we have

\[
\eta = a \frac{\Omega_0}{fg}
+ \lambda \frac{df}{f}
+ \mu \frac{dg}{g}
+ dh.
\]

The other cases are proven similarly.
\end{proof}

We say that two germs $f,g\in\mathcal{O}_m$ are in \emph{general position} if
$(df\wedge dg)(0)\neq 0$.
In particular, $f$ and $g$ are submersive at the origin $0\in \mathbb{C}^m$ and there is a local system of coordinates 
 $(x_1,x_2,..,x_m)\in \mathbb C^m,0$ with $x_1=f, x_2=g$. Two holomorphic functions 
 $f, g \in \mathcal O(U)$ for any open subset $U\subset \mathbb C^m$ are {\it in general position}
 if for every intersection point $\xi \in (f=0)\cap (g=0)$, the germms $f_\xi, g_\xi \in \mathcal O_{m, \xi}$ are in general position.

Analogously to Proposition~\ref{proposition:homology}, as in the proof of Proposition~\ref{proposition:cohomologyeqdimm} we can prove:
\begin{theorem}[Relative cohomology lemma]

Let \(\Omega \) be a germ of holomorphic 1-form at \(0 \in \mathbb C^m\).
Let \(f,g \in \mathcal O_m\) be germs in general position,
\(p,q \in \mathbb N\), \(\langle p,q \rangle = 1\).
\begin{enumerate}

\item Then the following conditions are equivalent:
\begin{enumerate}
\item
\[
d\Omega \wedge \Omega_0 = 0
\quad \text{and} \quad
\int_{\gamma_c} \Omega = 0,
\quad \forall \gamma_c \subset \phi_c.
\]

\item 
$\exists a,h \in \mathcal O_m
\quad \text{such that} \quad
\Omega = a \Omega_0 + dh.
$
\end{enumerate}

\item Similarly, the following conditions are equivalent:
\begin{enumerate}
\item 
\[
d\!\left(\frac{\Omega}{fg}\right) \wedge \Omega_0 = 0
\quad \text{and} \quad
\int_{\gamma_c} \frac{\Omega}{fg} = 0.
\]

\item
\[
\exists a,h \in \mathcal O_m
\quad \text{such that} \quad
\frac{\Omega}{fg}
=
a \frac{\Omega_0}{fg}
+ dh.
\]
\end{enumerate}
\end{enumerate}
\end{theorem}

\begin{proof}
The proof is similar to what we have already done: (1) reduction to dimension two, by  considering 
holomorphic embeddings $\tau \colon (\mathbb C^2, 0) \to (\mathbb C^m,0)$ in general position with respect to 
$\Omega_0=gdf - q f dg, f$ and $g$; (2) using Proposition~\ref{proposition:homology} we solve this local dimension $2$ case; (3) go back and solove the dimension $m \geq 2$ case, by applying  the extension theorem 
in \cite{[Cerveau--Berthier]}(Extension theorem p.405) considering the hypothesis $d \eta \wedge \Omega_0=0$ for $\eta=\Omega$ as in part (1) or $\eta=\Omega/fg$ as in part (2)).

\end{proof}
\section{First degree homogeneous deformations: the persistence theorem}

In this section we consider first degree deformations
$\Omega_t = \Omega_0 + t \Omega_1$
of
$
\Omega_0 = p g df - q f dg
$
by integrable 1-forms, where $t$ is a complex parameter.
Here, \(f,g\) are homogeneous polynomials in \(\mathbb C^m\),
\(p,q \in \mathbb N\), \(\langle p,q \rangle = 1\).

We have

\[
0 = \Omega_t \wedge d\!\left(\frac{\Omega_t}{fg}\right)
=
\left(\frac{\Omega_0 + t\Omega_1}{fg}\right)
\wedge
\left(d\frac{\Omega_0}{fg} + t\, d\frac{\Omega_1}{fg}\right).
\]

Since
\[
d\!\left(\frac{\Omega_0}{fg}\right)
=
d\!\left(\frac{p df}{f} - \frac{q dg}{g}\right)
=0,
\]
we obtain

\[
d\!\left(\frac{\Omega_1}{fg}\right) \wedge \Omega_0 = 0.
\]

Thus, by the previous proposition, we have

1st case \(p>1, q>1\):

\[
\frac{\Omega_1}{fg}
=
a \frac{\Omega_0}{fg}
+ \lambda \frac{df}{f}
+ \mu \frac{dg}{g}
+ dh.
\]

2nd case \(p=1<q\):

\[
\frac{\Omega_1}{fg}
=
a \frac{\Omega_0}{fg}
+ \lambda \frac{df}{f}
+ \mu \frac{dg}{g}
+ c \frac{g^{q-1}}{f} dg
+ dh.
\]

3rd case \(q=1<p\):

\[
\frac{\Omega_1}{fg}
=
a \frac{\Omega_0}{fg}
+ \lambda \frac{df}{f}
+ \mu \frac{dg}{g}
+ \tilde c \frac{f^{p-1}}{g} df
+ dh.
\]

Assume now that

\[
\int_{\gamma_c} \frac{\Omega_t}{fg} = 0,
\quad \forall \gamma_c \subset \phi_c.
\]

Then, on each case, we have
\[
\lambda q + \mu p = 0,
\quad
c=0,
\quad
\tilde c=0,
\quad
\alpha=\beta=0,
\]
the other cases being similar and we may state:
\begin{lemma}
For the integrable deformation $\Omega_t = \Omega_0 + t \Omega_1$, if 
$\int_{\gamma_c} \Omega_t=0, \forall \gamma_c \subset \phi_c, \forall c$ then 

\[
\frac{\Omega_t}{fg}
=
(1+\tilde a_t) \frac{\Omega_0}{fg}
+ dh.
\]

If, moreover, \(\Omega_t, f,g\) are homogeneous and
$
\deg \Omega_t = \deg \Omega_0 = \deg f + \deg g - 1,
$
then
\[
\Omega_t = c_t \Omega_0
\]
for some constant \(c_t \in \mathbb C\).
\end{lemma}
\begin{proof}
It only remains to prove the final part. Since 
$\deg \Omega_t = \deg \Omega_0 = \deg f + \deg g -1$, we conclude that 
$\tilde a_t$ is constant and $dh_t=0$. 
\end{proof}

\begin{proof}[Proof of Theorem~\ref{theorem:persistence}]

We first assume \(p>1, q>1\).
Then we may write

\[
\frac{\Omega_t}{fg}
=
a_t \frac{\Omega_0}{fg}
+ \lambda_t \frac{df}{f}
+ \mu_t \frac{dg}{g}
+ dh_t.
\]

Since \(\Omega_t\) is polynomial and \(\deg \Omega_t = \deg \Omega_0\),
we conclude that \(a_t\) is constant and \(dh_t=0\).

Hence

\[
\frac{\Omega_t}{fg}
=
a_t \frac{\Omega_0}{fg}
+ \lambda_t \frac{df}{f}
+ \mu_t \frac{dg}{g}.
\]

\[
=
a_t \left(p \frac{df}{f} - q \frac{dg}{g}\right)
+ \lambda_t \frac{df}{f}
+ \mu_t \frac{dg}{g}
\]

\[
=
(a_t p + \lambda_t)\frac{df}{f}
-
(a_t q - \mu_t)\frac{dg}{g}.
\]

Since we have \(0\) as a dicritical singularity of \(\Omega_t\),
we conclude that necessarily
\[
\frac{a_t p + \lambda_t}{a_t q - \mu_t}
=
\frac{m_t}{n_t}
\]
for some \(m_t, n_t \in \mathbb N\).

Thus

\[
\frac{\Omega_t}{fg}
=
c_t \left(
m_t \frac{df}{f}
-
n_t \frac{dg}{g}
\right)
=
c_t \, d \ln \left(\frac{f^{m_t}}{g^{n_t}}\right).
\]

Hence
\[
\phi_t = \frac{f^{m_t}}{g^{n_t}}
\]
is a rational first integral.
The other cases are quite similar, in view of Lemma~\ref{lemma:homologyglobal} and Propositions~\ref{proposition:homology} and ~\ref{proposition:cohomologyeqdimm}. 
\end{proof}


\section{Declarations:}

\noindent{\bf Funding}: Partial financial support was received from Fundação Getúlio Vargas (FGV-Rio de Janeiro), through the  The School of Applied Mathematics (EMAP-FGV).

\noindent{\bf Competing interests}: The authors have no competing interests to declare that are relevant to the content of this article.


\bibliographystyle{amsalpha}

\end{document}